\documentclass[9pt,twoside,english,reqno,a4paper]{amsart}
\usepackage{listings,graphicx,amsmath,varioref,amscd,amssymb,color,bm,stmaryrd,amsthm,amsfonts,graphics,geometry,latexsym,pgf,pst-all} 
\theoremstyle{plain}
\usepackage{esint}
\usepackage{amsthm}

\usepackage{hyperref}

\theoremstyle{plain}
\newtheorem{theorem}{Theorem}[section]
\newtheorem{proposition}[theorem]{Proposition}

\newtheorem{lemma}[theorem]{Lemma}

\usepackage{geometry}
\theoremstyle{definition}
\newtheorem{defin}[theorem]{Definition}

\newtheorem{remark}[theorem]{Remark}

\theoremstyle{remark}

\def\bk{\color{black}}

\numberwithin{equation}{section}

\def\Div{\operatorname{div}}
\def\into{\int_{\Omega}}

\DeclareMathOperator{\diver}{div}
\DeclareMathOperator{\sgn}{sgn}
\DeclareMathOperator{\R}{\mathbb{R}}

\newcommand{\car}[1]{\raise1pt\hbox{$\chi$}_{#1}}

\newcommand{\DM }{\mathcal{DM}^\infty }

\def\re{\mathbb{R}}

\newcommand{\res}{\!\!\mathop{\hbox{
			\vrule height 7pt width .5pt depth 0pt
			\vrule height .5pt width 6pt depth 0pt}}
	\nolimits}

\makeatletter
\@namedef{subjclassname@2020}{%
	\textup{2020} Mathematics Subject Classification}
\makeatother

\begin{document}
	\title[Competing gradient, singular and $L^1$ terms in nonlinear elliptic equations]{Finite energy solutions for nonlinear elliptic equations with competing gradient, singular and $L^1$ terms}

	\author[F. Balducci]{Francesco Balducci}
	\author[F. Oliva]{Francescantonio Oliva}
	\author[F. Petitta]{Francesco Petitta}
	\address{Francesco Balducci
		\hfill \break\indent
		Dipartimento di Scienze di Base e Applicate per l' Ingegneria, Sapienza Universit\`a di Roma
		\hfill \break\indent
		Via Scarpa 16, 00161 Roma, Italy}
	\email{\tt francesco.balducci@uniroma1.it}
	\address{Francescantonio Oliva
		\hfill \break\indent
		Dipartimento di Scienze di Base e Applicate per l' Ingegneria, Sapienza Universit\`a di Roma
		\hfill \break\indent
		Via Scarpa 16, 00161 Roma, Italy}
	\email{\tt francescantonio.oliva@uniroma1.it}
	\address{Francesco Petitta
		\hfill \break\indent
		Dipartimento di Scienze di Base e Applicate per l' Ingegneria, Sapienza Universit\`a di Roma
		\hfill \break\indent
		Via Scarpa 16, 00161 Roma, Italy}
	\email{\tt francesco.petitta@uniroma1.it}


	\begin{abstract}
		
		In this paper we deal with  the  following boundary value problem
		\begin{equation*}
			\begin{cases}
				-\Delta_{p}u + g(u) | \nabla u|^{p} = h(u)f & \text{in $\Omega$,} \\
				u\geq 0 & \text{in $\Omega$,} \\
				u=0 & \text{on $\partial \Omega$,} \
			\end{cases}
		\end{equation*}
		in a domain  $\Omega \subset \mathbb{R}^{N}$ $(N \geq 2)$, where $1\leq p<N $,  $g$ is  a positive and continuous function on $[0,\infty)$,  and $h$ is a continuous function on $[0,\infty)$ (possibly blowing up at the origin). We show how the presence of regularizing terms $h$ and $g$ allows to prove existence of finite energy solutions for nonnegative  data $f$ only belonging to $L^1(\Omega)$.  
	\end{abstract}
	
	\maketitle
	\tableofcontents
	
	\section{Introduction}
	
	The aim of this work is the study of the following boundary value problem
	\begin{equation}\label{problemone}
		\begin{cases}
			-\Delta_{p}u + g(u) | \nabla u|^{p} = h(u)f & \text{in $\Omega$,} \\
			u \ge 0 & \text{in $\Omega$,} \\
			u=0 & \text{on $\partial \Omega$,} \
		\end{cases}
	\end{equation}
	where $1\le p <N$.   Here $\Omega \subset \mathbb{R}^{N}$,  with $N \ge 2$,  is an open and bounded set with Lipschitz boundary,  $\Delta_{p} u:=\operatorname{div}(|\nabla u|^{p-2} \nabla u)$ is the usual  $p-$laplacian,   and  $g$ is  a positive and continuous function on $[0,\infty)$. Finally $h$ is a continuous function on $[0,\infty)$ that is allowed to blow-up at the origin and is bounded at infinity. In particular,  the case of continuous, bounded and non-monotone functions $g,h$ is covered by the above assumptions.
	
	\medskip
	
	Our goal is the study of existence of finite energy solutions to 
	\eqref{problemone}; by a finite energy solution we mean a function lying in the natural space in which such problems are naturally built-in in case of smooth nonlinear terms and data, i.e. $u \in \textit{W}^{1,p}_{0}(\Omega)$ if $p>1$ and 
	$u \in \textit{BV}(\Omega)$ if $p=1$.

	We are interested to deeply explore the interplay between the first order absorption term involving $g$ and the zero order and possibly singular nonlinearity $h$ in presence of a merely integrable datum $f$. In particular we deal with the regularizing effect,  in terms of Sobolev regularity,  provided by the lower order terms to the solutions of problems as  \eqref{problemone}. 
	
	\medskip

	These kinds of regularizing effects given by the gradient terms  with natural growth in elliptic problems with rough data are nowadays quite classical. If $h\equiv 1$ it is well known that, see for instance \cite{bbm,po}, quadratic gradient terms satisfying a sign condition (i.e. $g(s)s\geq 0$)  in problems as \eqref{problemone}  gives finite energy solutions if $f$ is a merely integrable function (or even a measure). Let also stress that it is well known that solutions have, in general, infinite energy if $g\equiv 0$ (see \cite{b6}). 
	
	\medskip 
	
	On the other hand problems as  $(\ref{problemone})$  with  $g \equiv 0$ and in presence of possibly singular nonlinearities has reached great attention in the last decades starting from the pioneering papers \cite{crt,lm,st}; if $p=2$, $h(s)=s^{-\gamma}$ $(\gamma>0)$,  and $f$ is smooth, existence of classical  solutions in  $C^2(\Omega)\cap C^{0} (\overline{\Omega})$ follows by suitable approximation with desingularized problems. Only later, in \cite{11}, the authors prove existence of a distributional solution in case of a Lebesgue datum $f$ and remarked the regularizing effect given by the right-hand of $(\ref{problemone})$ when, once again, $p=2$ and $h(s)=s^{-\gamma}$ $(\gamma>0)$: namely the solution, compared to the case $\gamma =0$,  always lies in a smaller Sobolev space when $0<\gamma \le 1$. Moreover if $\gamma=1$ the solution is always in ${H}^{1}_{0}(\Omega)$ even if $f$ is just an $\textit{L}^{1}-$function as one can formally deduce by taking $u$ itself as test function in $(\ref{problemone})$ while if $\gamma > 1$ the solution belongs only locally to ${H}^{1}(\Omega)$ and the boundary datum is meant as a suitable power of the solution having zero Sobolev trace. {It is also worth to mention that, in general, one can not expect finite energy solutions for $\gamma\geq 3-\frac{2}{m}$ if $f\in L^m(\Omega)$, $m> 1$ (\cite{lm,op}).}

	\medskip 
	
	Similar results, again in the case $g \equiv 0$,  were then extended to the case $p\geq 1$; i.e. let us consider nonnegative finite energy solutions for
	$$
	\begin{cases}	
		\displaystyle -\Delta_{p}u  = \frac{f}{u^\gamma} & \text{in $\Omega$,} \\
		u=0 & \text{on $\partial \Omega$.} \	
	\end{cases}
	$$
	It was shown in \cite{dc} that, if $p>1$,  then finite energy solutions   exist either,  for smooth datum $f$,  up to $\gamma<2+\frac{1}{p-1}$ {or} if $\gamma \leq 1$ with  $f\in L^{m}(\Omega)$ for some $m\geq 1$. As observed in \cite{dgop}, the formal case of   $p\to 1^{+}$ is also included as finite energy BV solutions are expected to exist either for any $\gamma>0$ in case of smooth $f$, or for $\gamma\leq 1$ in the general case of $f\in L^{m}(\Omega)$ for some $m\geq 1$.  
	\medskip 
	
	Problems as 
	$$
	\begin{cases}
		-\Delta_{1}u + g(u) | D u| = f & \text{in $\Omega$,} \\
		u=0 & \text{on $\partial \Omega$,} \
	\end{cases}
	$$
	have also been recently considered as a model for the level set formulation proposed in \cite{hi} for the inverse mean curvature flow (see also \cite{ms}) in order to prove the well-known Penrose inequality in the case of a single black hole.

	From the purely mathematical point of view,   regularizing effect of the presence of  gradient terms (also in case of a nonlinearity  $g$ with a generic sign) and $f\in L^N(\Omega)$ was recently investigated (see \cite{4, ms, ls, gop}). Among the others let only stress that solutions to these problems are, in general,  $BV$ functions with no jump part in its gradient. 
	
	\medskip
	
	The {main} goal of our study concerns  problems as 
	\begin{equation}\label{pb1}
		\begin{cases}
			-\Delta_{1}u + g(u) | D u|  = h(u)f & \text{in $\Omega$,} \\
			u=0 & \text{on $\partial \Omega$,} \
		\end{cases}
	\end{equation}
	where $g$ is a positive and bounded continuous function on $[0,\infty)$, $h$ is a continuous function on $[0,\infty)$ (possibly blowing up at the origin) and bounded at infinity, and $f$ is a nonnegative function in $L^{1}(\Omega)$. In particular we will focus on the interplay between the data $g, h$ and $f$ in order to get nontrivial solutions in $BV(\Omega)$, the space of functions in $L^{1}(\Omega)$ whose derivatives have finite total bounded variation over $\Omega$. 
	We stress again that this goal is forbidden in presence  of merely integrable data   whenever  $g\equiv 0$ (see  \cite{lops}). It is also worth mentioning that, if this is the  case, existence of solutions is expected  for more regular data, but only under a suitable smallness condition on the norm (\cite{dgs, dgop}),  which here is not requested.  
	
	\medskip
	
	As a further feature of these types of problems with gradient type absorption terms one has  that solutions $u$ of \eqref{pb1} are in fact "zero" $\mathcal{H}^{N-1}$a.e. on $\partial\Omega$  in contrast with the case of reaction terms (\cite{ls,gop}) or no reaction at all (\cite{dgop}) in which constant solutions are allowed. 
	\medskip 
	
	We also stress that a standard approach to deal with $1$-Laplace type problems as \eqref{pb1} consists in approximate them  with $p$-Laplace ones where $p>1$. 
	Our method will follow this line but let us emphasize that the case $p>1$ is interesting and new as well; so that, in Section \ref{sec:p>1} we set, as a preparatory tool but in full generality,  the theory of existence and weak regularity of solutions for problems as in \eqref{problemone}. 
	
	\medskip
	
	The plan of the paper is the following: in Section \ref{sec:prel} we provide some preliminaries tools and notation. In Section \ref{sec:p>1} we study the Dirichlet problem in presence of a principal operator modeled by the $p$-Laplacian with $p>1$; apart from being interesting itself, it is the preparatory study for the main Section \ref{sec:p=1}, in which the limit case $p=1$ is investigated.  Finally in Section \ref{sec:ex}, we give some further insights on  how to extend the result in various directions as the case of merely nonnegative data  and the case of a nonlinearity $g$ possibly blowing-up at infinity. 
	
	\section{Notation and preparatory tools}
	\label{sec:prel}
	
	For the entire paper $\Omega$ denotes an open bounded subset of $\R^N$, for $N\ge 2$,  with Lipschitz boundary. We stress that for the results of Section \ref{sec:p>1} the Lipschitz regularity is not needed.
	By $\mathcal H^{N-1}(E)$ we mean the $(N - 1)$-dimensional Hausdorff measure of a set $E$ while $|E|$ stands for its $N$-dimensional Lebesgue measure. 
	The space $\mathcal{M}(\Omega)$ is the usual one of Radon measures with finite total variation over $\Omega$. Its local counterpart $\mathcal{M}_{\rm loc}(\Omega)$ is the space of Radon measures which are locally finite in $\Omega$.
	
	\medskip
	
	For a fixed $k>0$, we use the truncation functions $T_{k}:\R\to\R$ and $G_{k}:\R\to\R$ defined by
	\begin{align}\label{trunc}
		T_k(s):=&\max (-k,\min (s,k))\ \ \text{\rm and} \ \ G_k(s):=s- T_k(s).
	\end{align}
	Moreover we define
	\begin{align}\label{Vdelta}
		\displaystyle
		V_{\delta}(s):=
		\begin{cases}
			1 \ \ &  0\le s\le \delta, \\
			\displaystyle\frac{2\delta-s}{\delta} \ \ &\delta <s< 2\delta, \\
			0 \ \ &s\ge 2\delta.
		\end{cases}
	\end{align}

	Finally, for a Banach space $X$ we denote by $C_b^0(X)$  the space of bounded and continuous real functions on $X$.

	If no otherwise specified, we denote by $C$ several positive constants whose value may change from line to line and, sometimes, on the same line. These values will only depend on the data but they will never depend on the indexes of the sequences we will gradually introduce. Let us explicitly mention that we will not relabel an extracted compact subsequence.

	For simplicity's sake, and if there is no ambiguity,  we will often use the following notation:
	$$
	\int_\Omega f:=\int_\Omega f(x)\ dx.
	$$
	
	\subsection{Basics on $BV$ spaces} 
	The Banach space of bounded variation functions on $\Omega$ is defined as:
	$$BV(\Omega):=\{ u\in L^1(\Omega) : Du \in \mathcal{M}(\Omega)^N \}\,,$$ 
	endowed with the norm  
	$$\displaystyle ||u||_{BV(\Omega)}=\int_{\partial\Omega}
	|u|\, d\mathcal H^{N-1}+ \int_\Omega|Du|\,,$$
	where $|Du|$ denotes the total variation of the measure $D u$. 
	With $L_u$ we mean the set of Lebesgue points of a function $u$, with $S_u :=\Omega\setminus L_u$ and  with $J_u$ the  jump set. Let recall that any function $u\in BV(\Omega)$ can be identified with its precise representative $u^*$ which is the Lebesgue representative in $L_u$ while $u^*=\frac{u^+ +u^-}{2}$ in $J_u$ where $u^+,u^-$ denote the approximate limits of $u$. Moreover it can be shown that 
	$\mathcal{H}^{N-1}(S_u\setminus J_u)=0$ and that $u^*$ is well defined and finite $\mathcal{H}^{N-1}$-a.e. 
	
	\medskip
	
	Let us highlight that, once saying that $D^j u=0$, we understand that $\mathcal{H}^{N-1}( J_u)=0$, or, in other terms, that $Du=\tilde{D}u$ where $\tilde{D} u$ is the absolutely continuous part of $D u$ with respect to the Lebesgue measure. In this case we will also denote by $u$ instead of $u^*$ the  precise representative of $u$ as no ambiguity is possible when integrating against a measure which is absolutely continuous with respect to $\mathcal{H}^{N-1}$.

	\subsection{The Anzellotti-Chen-Frid theory} In order to be self-contained we  summarize the $L^\infty$-divergence-measure vector fields theory due to \cite{An} and \cite{CF}. We denote by 
	$$\DM(\Omega):=\{ z\in L^\infty(\Omega)^N : \operatorname{div}z \in \mathcal{M}(\Omega) \},$$
	and by $\DM_{\rm loc}(\Omega)$ its local version, namely the space of bounded vector field $z$ with $\operatorname{div}z \in \mathcal{M}_{\rm loc}(\Omega)$.
	We first recall that if $z\in \DM(\Omega)$ then $\operatorname{div}z $ is an absolutely continuous measure with respect to $\mathcal H^{N-1}$. 
	\\In \cite{An} the following distribution $(z,Dv): C^1_c(\Omega)\to \mathbb{R}$ is introduced: 
	\begin{equation}\label{dist1}
		\langle(z,Dv),\varphi\rangle:=-\int_\Omega v^*\varphi\operatorname{div}z-\int_\Omega
		vz\cdot\nabla\varphi,\quad \varphi\in C_c^1(\Omega),
	\end{equation}
	in order to define a generalized pairing between vector fields in  $\DM(\Omega)$ and derivatives of $BV$ functions. 
	In \cite{MST2} and \cite{C}, in fact,  the authors prove that $(z, Dv)$ is well defined if $z\in \DM(\Omega)$ and $v\in BV(\Omega)\cap L^\infty(\Omega)$ since one can show that $v^*\in L^\infty(\Omega,\operatorname{div}z)$. Moreover in \cite{dgs} the authors show that \eqref{dist1} is well posed if $z\in \DM_{\rm loc}(\Omega)$ and $v\in BV_{\rm loc}(\Omega)\cap L^1_{\rm loc}(\Omega, \operatorname{div}z)$ and it holds that
	\begin{equation*}\label{finitetotal}
		|\langle   (z, Dv), \varphi\rangle| \le ||\varphi||_{L^{\infty}(U) } ||z||_{L^\infty(U)^N} \int_{U} |Dv|\,,
	\end{equation*}
	for all open sets $U \subset\subset \Omega$ and for all $\varphi\in C_c^1(U)$. Moreover one has
	\begin{equation}\label{finitetotal1}
		\left| \int_B (z, Dv) \right|  \le  \int_B \left|(z, Dv)\right| \le  ||z||_{L^\infty(U)^N} \int_{B} |Dv|\,,
	\end{equation}
	for all Borel sets $B$ and for all open sets $U$ such that $B\subset U \subset \Omega$.

	\medskip

	Observe that,  if $z\in \DM_{\rm loc}(\Omega)$ and $w\in BV_{\rm loc}(\Omega)\cap L^\infty(\Omega)$, then
	\begin{equation}\label{27}
		\diver (wz)=(z,Dw) + w^*\diver z\,,
	\end{equation}
	so that $wz \in \DM_{\rm loc}(\Omega)$.

	We recall that in \cite{An} it is proved that every $z \in \mathcal{DM}^{\infty}(\Omega)$ possesses a weak trace on $\partial \Omega$ of its normal component which is denoted by
	$[z, \nu]$, where $\nu(x)$ is the outward normal unit vector defined for $\mathcal H^{N-1}$-almost every $x\in\partial\Omega$. Moreover, it holds that 
	\begin{equation}\label{des1}
		||[z,\nu]||_{L^\infty(\partial\Omega)}\le ||z||_{L^\infty(\Omega)^N},
	\end{equation}
	and  also that,  if $z \in \mathcal{DM}^{\infty}(\Omega)$ and $v\in BV(\Omega)\cap L^\infty(\Omega)$, then
	\begin{equation*}\label{des2}
		v[z,\nu]=[vz,\nu],
	\end{equation*}
	(see \cite{C}).\\
	Finally  the following Green formula holds (see \cite{dgs}).
	\begin{lemma}\label{poiu}
		Let $z \in \mathcal{DM}_{\rm{loc}}^{\infty}(\Omega)$ and set $\mu=\Div z$. Let $v\in BV(\Omega)\cap L^\infty(\Omega)$ be such that $v^*\in L^1(\Omega,\mu)$.
		Then $vz\in \mathcal{DM}^{\infty}(\Omega)$ and the following  holds: 
		\begin{equation}\label{green}
			\int_{\Omega} v^* \, d\mu + \int_{\Omega} (z, Dv) =
			\int_{\partial \Omega} [vz, \nu] \ d\mathcal H^{N-1}\,.
		\end{equation}
	\end{lemma}

	Analogously to \eqref{des1}, it can be proved (see \cite[Proposition 2.7]{dgs}) that, for $z\in \DM_{\rm{loc}}(\Omega)$ such that the product $vz\in \DM(\Omega)$ for some $v\in BV(\Omega)\cap L^\infty(\Omega)$,
	\begin{equation}\label{des3}|[vz,\nu]|\le |v_{\res {\partial\Omega}}|\,\|z\|_{L^\infty(\Omega)^N}\,\quad\mathcal H^{N-1}\hbox{-a.e. on }\partial\Omega.\end{equation}

	When dealing with compositions with nonlinear functions it is sometimes useful to define  a slightly  "different" pairing measure as follows (see for instance \cite{ms}):  let $\beta:\re\to \re$ be a locally Lipschitz  function and $v \in BV_{\rm loc}(\Omega)$, then  we define 
	\begin{equation*}\displaystyle \beta(v)^{\#} :=\begin{cases} \displaystyle \frac{1}{v^{+} -v^-}\int_{v^-}^{v^+} \beta(s) \ ds  & \text{if}\ \ x\in J_v,\\
			\beta(v) & {\text{otherwise}.}
			\ \end{cases}\label{tilde} \end{equation*}
	
	Observe that  	 $\beta(v)^{\#}$	turns out to coincide with  $\beta(v)^{*}$ on the jump set if and only if $\beta(s)=s$. 
	As for \eqref{dist1}, if $z \in \mathcal{DM}_{\rm{loc}}^{\infty}(\Omega)$ and $v\in BV_{\rm loc} (\Omega)$ is such that $\beta(v)\in BV_{\rm loc}(\Omega)\cap L^{\infty}_{\rm loc}(\Omega)$, it is possible to define the measure $(z,D\beta(v)^{\#})$ by 
	\begin{equation}\label{dist1d}
		\langle(z,D\beta(v)^{\#}),\varphi\rangle:=-\int_\Omega \beta(v)^{\#}\varphi\operatorname{div}z-\int_\Omega
		\beta(v)z\cdot\nabla\varphi,\quad \varphi\in C_c^1(\Omega).
	\end{equation}

	By \cite[Lemma $2.5$]{ms}, this new  pairing $(z, D\beta(v)^{\#})$ is a well defined measure absolutely continuous with respect to  $\mathcal{H}^{N-1}$,  and moreover
	\begin{equation}\label{finitetotal1d}
		\left| \int_B (z, D\beta (v)^{\#}) \right|  \le  \int_B \left|(z, D\beta (v)^{\#})\right| \le  ||z||_{L^\infty(U)^N} \int_{B}  |D\beta(v)|,
	\end{equation}
	for all Borel sets $B$ and for all open sets $U$ such that $B\subset U \subset \Omega$.
	
	\medskip 
	
	In what follows we will use the classical chain rule formula for functions in $BV$ (\cite[Theorem 3.99]{afp}).
	
	\begin{lemma}\label{chainrule}
		Let $u\in BV(\Omega)$  and let $\Psi: \re\to \re$ be a Lipschitz function.  Then $v=\Psi(u) \in BV(\Omega)$ and
		\begin{equation}\label{cr}
			Dv = \Psi'(u)^{\#} Du. 
		\end{equation}
		
		In particular, if  $D^j u=0$, then  $\tilde{D}v= \Psi'(u)\tilde{D}u$. 
	\end{lemma}
	\bk 
	\subsection{A general result on the jump part of a $BV$ function}
	
	We show  a general result showing that a function  $w\in BV_{\rm loc}(\Omega)$ satisfying an inequality involving its gradient  has no jump part. The proof is a suitable re-adaptation  of an idea in \cite{gmp}.

	\begin{lemma}\label{lemma derivata salto nulla}
		Let $\alpha$ and $\beta$ be two locally Lipschitz increasing  functions on $\mathbb{R}$. Let $z\in \DM_{\rm{loc}}(\Omega)$ such that $\|z\|_{L^{\infty}(\Omega)^N}\leq 1$, $w,\alpha(w)\in BV_{\rm loc}(\Omega)$, $
		\beta(w)\in BV_{\rm loc}(\Omega)\cap L^{\infty}_{\rm loc}(\Omega)$, and $\lambda \in L^1_{\rm loc}(\Omega)$. Moreover assume that 
		\begin{equation}\label{241}
			-\operatorname{div}z+|D \alpha(w)| \le  \lambda 
		\end{equation}
		and
		\begin{equation}\label{242}
			(z, D \beta(w)^{\#}) = |D \beta(w)|\,,
		\end{equation}
		as measures. Then $D^{j}w=0$. 
	\end{lemma}
	\begin{proof}
		Since $\beta(w) \in BV_{\rm{loc}}(\Omega)$, by  \cite[Theorem 3.78]{afp}, $S_{\beta(w)}$ is a (locally) countably $\mathcal{H}^{N-1}-$rectifiable set and then there exist regular hypersurfaces $\xi_{i}$ such that $\mathcal{H}^{N-1}\left( S_{\beta(w)} \setminus \bigcup_{i=1}^{\infty} \xi_{i} \right)=0$.
		
		By monotonicity of $\beta$, the proof  follows  once we prove that $|D\beta( w)|(\xi_{i})=0$ for any $i \in \mathbb{N}$.
		
		First observe that by \cite[Proposition $3.4$]{ACM2} one has  that
		\begin{equation}\label{nuova1}
			\operatorname{div}z = [z,\nu]^+- [z,\nu]^- \ \text{on} \ \xi_{i},
		\end{equation}
		where $[z,\nu]^+$ and $[z,\nu]^-$ are the traces of the normal components of $z$ over  $\xi_{i}$ defined as in \cite[Definition $3.3$]{ACM2}. 
		
		
		\medskip 
		Hence,  using  \cite[Corollary $3.2$]{cdc}, one has that
		\begin{equation}\begin{aligned}\label{nuova2}
				(z,D \beta(w)^{\#}) &= - \beta(w)^{\#}\operatorname{div}z + \operatorname{div}(\beta(w)z) \\ &= 
				{[z,\nu]^+  (\beta(w)^+ - \beta(w)^{\#}) + [z,\nu]^-}(\beta(w)^{\#}-\beta(w)^-) \le |\beta(w)^+-\beta(w)^-| \ \text{on} \ \xi_{i}, 
		\end{aligned}	\end{equation}
		with the equality sign if and only if 	$[z,\nu]^+ =[z,\nu]^- =\sgn(\beta(w)^+-\beta(w)^-) $. 
		
		Now, as 
		\begin{equation*}
			|D\beta(w)|= |\beta(w)^+-\beta(w)^-| \ \text{on} \ \xi_{i},
		\end{equation*}
		from \eqref{242} we then get  the equality sign in  \eqref{nuova2}; so that 
		\begin{equation}\label{nuova3}
			[z,\nu]^+= [z,\nu]^- =\sgn (\beta(w)^+-\beta(w)^-) \ \text{on} \ \xi_{i}.
		\end{equation}
		Observe that both $\alpha(w)$ and  $\beta(w)$ share the same jump set so that   we gather together  \eqref{241}, \eqref{nuova1} and \eqref{nuova3} to obtain  that 
		$$|\alpha(w)^+ - \alpha (w)^-|=|D\alpha (w)| = 0 \ \text{on} \ \xi_{i},$$
		which allows us to conclude that  
		$$|D \alpha(w)| = 0 \ \text{on} \ \xi_{i},$$
		and, as   $\alpha$ is strictly monotone, that  finally  $D^{j}w=0$. 
	\end{proof}
	\bk

	\section{The case $p>1$ for general monotone operators}
	\label{sec:p>1} 
	
	As we mentioned, in this section we set, not only as a preparatory tool,  the theory of existence and weak regularity of solutions for problems as in \eqref{problemone} for Leray-Lions nonlinear operators as leading terms.
	
	\medskip 
	
	Let $\Omega \subset \mathbb{R}^{N}$ $(N \ge 2)$ is an open and bounded set (no further regularity is needed here),  and let  $a: \Omega \times \mathbb{R} \times \mathbb{R}^{N} \to \mathbb{R}^{N} $ be a Carath\'eodory function {satisfying} the classical  Leray-Lions assumptions, i.e. there exist $\alpha,\beta>0$ and  $1<p<N$ such that, for almost every $x \in \Omega$, for every $s \in \mathbb{R}$ and every $\xi, \xi' \in \mathbb{R}^{N}$:
	\begin{equation}\label{cara1}
		a(x,s,\xi) \cdot   \xi \ge \alpha |\xi|^{p},
	\end{equation}
	\begin{equation}\label{cara2}
		|a(x,s, \xi)|  \le \beta \left( b(x)+|s|^{p-1} + |\xi|^{p-1} \right), 
	\end{equation}
	\begin{equation}\label{cara3}
		\left(a(x,s,\xi) - a(x,s, \xi')\right)  \cdot (\xi - \xi') > 0;
	\end{equation}
	let us recall that, by \eqref{cara1}, one has $a(x,s,0)=0$, for any $s\in \re$ and a.e. $x\in\Omega$. 
	
	Consider the following boundary value problem
	\begin{equation}\label{problema}
		\begin{cases}
			-\operatorname{div}(a(x,u,\nabla u)) + g(u) | \nabla u|^{p} = h(u)f & \text{in $\Omega$,} \\
			u \ge 0 & \text{in $\Omega$,} \\
			u=0 & \text{on $\partial \Omega$,} \
		\end{cases}
	\end{equation}
	where $g: [0,\infty)\to [0,\infty)$ is a
	continuous function such that 
	\begin{equation}\label{g infinito}
		\liminf_{s \to \infty} g(s)>0. 
	\end{equation}
	Moreover $h :[0,\infty)\to [0,\infty]$ is a  continuous function such that $h(0)>0$,  
	\begin{equation}\label{h vicino 0}
		\exists \ 0 \le \gamma \le 1, c_{1},s_{1}>0: h(s) \le \frac{c_{1}}{s^{\gamma}} ~~~\text{for all $s \le s_{1}$}, 
	\end{equation}
	and
	\begin{equation}\label{h infinito}
		h\in C^0_b([\delta,\infty))\ \ \ \forall \delta>0. 
	\end{equation}
	
	Let us stress that the classical case $g,h \equiv 1$ is  covered by the above assumptions as $\gamma$ could be zero and $g$ does not necessarily satisfy $g(0)=0$.
	
	\medskip
	
	As for the datum, we assume that $f\in L^{1}(\Omega)$ is nonnegative. 
	
	\medskip 
	
	Let us introduce the notion of distributional solution for problem $(\ref{problema})$.
	\begin{defin}\label{soluzione p}
		Let $1<p<N$ then a nonnegative $u\in W^{1,p}_0(\Omega)$ is a distributional solution to \eqref{problema} if $a(x,u,\nabla u)\in L^1_{\rm loc}(\Omega)^N$, $g(u)|\nabla u|^{p}, h(u)f \in L^{1}_{\rm{loc}}(\Omega)$ and if it holds
		\begin{equation}\label{def_p>1}
			\int_{\Omega} a(x,u,\nabla u) \cdot \nabla \varphi + \int_{\Omega} g(u) |\nabla u|^{p} \varphi = \int_{\Omega} h(u) f \varphi, ~~  \forall \varphi \in C^{1}_{c}(\Omega).
		\end{equation}
	\end{defin}
	\begin{remark}
		It is worth mentioning that all the terms appearing in the weak formulation of Definition \ref{soluzione p} are well defined.
		In particular, as the vector field $a$ satisfies \eqref{cara1}-\eqref{cara3}, then $a(x,u,\nabla u)\in L^{p'}(\Omega)^N$  since $u\in W^{1,p}_0(\Omega)$, where $p'=\frac{p}{p-1}$ is the standard H\"older conjugate exponent of $p$. 
	\end{remark}

	Now we are ready to state the main result of this section. 
	\begin{theorem}\label{teorema p>1}
		Let $a$ satisfy  \eqref{cara1}-\eqref{cara3} with $1<p<N$. Let $g$ satisfy \eqref{g infinito} and let $h$ satisfy  \eqref{h vicino 0}-\eqref{h infinito}. Finally let $f \in L^{1}(\Omega)$ be nonnegative. Then there exists a solution $u$ to \eqref{problema} in the sense of Definition \ref{soluzione p} such that $g(u)|\nabla u|^p\in L^1(\Omega)$.
	\end{theorem}
	
	\subsection{Approximation scheme} 
	
	In order to prove Theorem $\ref{teorema p>1}$ we work by approximation. We will show the existence of a nonnegative  solution for the following 
	\begin{equation}\label{problema troncato}
		\begin{cases}
			-\operatorname{div}(a(x,u_{n}, \nabla u_{n})) + g_{n}(u_{n})|\nabla u_{n}|^{p} = h_{n}(u_{n})f_{n} & \text{in $\Omega$,} \\
			u_{n}=0 & \text{on $\partial \Omega$,}
		\end{cases}
	\end{equation}
	where $$g_{n}(s):=\begin{cases}
		\min\{n, g(0)\} & \text{if $s \le 0$, } \\
		T_{n}(g(s)) & \text{if $s>0$,}
	\end{cases}$$ $f_{n}:=T_{n}(f)$ and $h_{n}:=T_{n}(h)$ with $n>0$ ($T_n(s)$ is defined in \eqref{trunc}). 
	
	\medskip 
	The proof is standard and  we will only  sketch it for the sake of completeness. \bk 
	\begin{lemma}\label{lem_esistenzaun}
		Let $a$ satisfy \eqref{cara1}-\eqref{cara3} with $1<p<N$, {then} there exists a nonnegative $u_{n} \in W^{1,p}_{0}(\Omega) \cap L^{\infty}(\Omega)$ which satisfies $$\int_{\Omega} a(x,u_{n}, \nabla u_{n}) \cdot \nabla \varphi + \int_{\Omega} g_{n}(u_{n}) |\nabla u_{n}|^{p} \varphi = \int_{\Omega}  h_{n}(u_{n})f_{n} \varphi, \ \ \forall\varphi \in W^{1,p}_{0}(\Omega) \cap L^{\infty}(\Omega).$$ 
	\end{lemma}
	
	\begin{proof}[Sketch of the proof]
		For simplicity we assume that,  for any $v \in L^{p}(\Omega)$, problem 
		\begin{equation}\label{problema apprissimante}
			\begin{cases}
				-\operatorname{div}(a(x,w, \nabla w)) + g_{n}(w)|\nabla w|^{p} = h_{n}(|v|)f_{n} & \text{in $\Omega$,} \\
				w=0 & \text{on $\partial \Omega$,}
			\end{cases}
		\end{equation}
		admits a unique bounded solution. Existence is guaranteed  from \cite[Theorem 1]{bmp} in full generality. Concerning  uniqueness, it holds under quite natural additional assumptions on $a$ and $g$ (see for instance \cite{seg,abm,lpr,art} and references therein). \bk 
		\smallskip	
		
		We set
		$$S: L^{p}(\Omega) \to L^{p}(\Omega),$$ the map that, for any $v \in L^{p}(\Omega)$, gives the weak solution $w$ to \eqref{problema apprissimante}. 
		In particular $w \in W^{1,p}_{0}(\Omega)\cap L^{\infty}(\Omega)$ satisfies
		\begin{equation}\label{debole_schauder}
			\int_{\Omega} a(x,w, \nabla w) \cdot \nabla \varphi + \int_{\Omega} g_{n}(w)|\nabla w|^{p} \varphi = \int_{\Omega} h_{n}(|v|)f_{n} \varphi, \ \ \forall \varphi \in W^{1,p}_{0}(\Omega) \cap L^{\infty}(\Omega).
		\end{equation}
		
		We claim that $w$ is  nonnegative; indeed, one can fix in \eqref{debole_schauder} $\varphi=-w^{-} e^{-t w}$, for some  $t>0$ to be chosen later, and, here,  {$w^-\ge 0$} denotes  the negative part of $w$. This yields to
		\begin{equation*}
			\begin{aligned}
				& -\int_{\Omega} a(x,w, \nabla w) \cdot \nabla w^{-} e^{-t w}  + \int_{\Omega} t a(x, w, \nabla w) \cdot \nabla w  w^{-} e^{-t w} - \int_{\Omega} g_{n}(w) |\nabla w|^{p} w^{-} e^{- t w} 
				\\
				&= - \int_{\Omega} h_{n}(|v|) f_{n} w^{-} e^{- t w} \le 0,
			\end{aligned}
		\end{equation*}
		which, recalling \eqref{cara1}, implies 
		$$\alpha \int_{\Omega} | \nabla w^{-}| ^{p} e^{- t w} + \int_{\Omega} |\nabla w|^{p} w^{-} e^{- t w} (\alpha t - g_{n}(w)) \le 0. $$ 
		Hence it is sufficient requiring $t>\frac{n}{\alpha}$ in order to deduce that $w\ge 0$ almost everywhere in $\Omega$.
		
		\medskip
		
		Now we show that the map $S$ has an invariant ball, it is continuous and relatively compact in $L^p (\Omega)$, so that the Schauder fixed point Theorem can be applied.
		
		\medskip
		
		Let us fix $\varphi=w$ in \eqref{debole_schauder} yielding to
		\begin{equation}\label{disuguaglianza utile}
			\alpha \int_{\Omega} |\nabla w|^{p} \stackrel{\eqref{cara1}}{\le} \int_{\Omega} a(x,w, \nabla w) \cdot \nabla w + \int_{\Omega} g_{n}(w)|\nabla w|^{p} w = \int_{\Omega} h_{n}(|v|)f_{n}w \le n^{2} |\Omega|^{\frac{1}{p'}} \|w\|_{L^{p}(\Omega)},
		\end{equation}
		after an application of the H\"older inequality on the right-hand.
		Using the Poincar\'e inequality on the left-hand, one deduces  
		\begin{equation*}
			\| w\|_{L^{p}(\Omega)} \le \left(\frac{c^p(p,\Omega) n^{2}}{\alpha}\right)^{\frac{1}{p-1}}|\Omega|^{\frac{1}{p}},
		\end{equation*}
		where $c(p,\Omega)$ is the Poincar\'e  constant; observe that 
		these estimates are independent of $v$. Thus, we can affirm that the ball of radius $\left(\frac{c^p(p,\Omega) n^{2}}{\alpha}\right)^{\frac{1}{p-1}}|\Omega|^{\frac{1}{p}}$ is invariant for $S$. 
		
		\medskip
		
		Moreover from $(\ref{disuguaglianza utile})$ one deduces that 
		\begin{equation}\label{G nella palla}
			\|w\|_{W^{1,p}_{0}(\Omega)}\le C,
		\end{equation} 
		where $C$ is independent of $v$. This is sufficient to deduce that $S(L^{p}(\Omega))$ is relatively compact in $L^{p}(\Omega)$ by Rellich-Kondrachov Theorem.
		
		\medskip
		
		It is left to show that $S$ is continuous in $L^{p}(\Omega)$. Let consider  $v_{k}\in L^{p}(\Omega)$ converging to $v \in L^{p}(\Omega)$ as $k\to \infty$.
		
		If we denote  by  $w_{k}=S(v_{k})$ then $w_k$ is bounded in $W^{1,p}_{0}(\Omega)$ with respect to $k$ thanks to $(\ref{G nella palla})$. Moreover, exists $w \in W^{1,p}_{0}(\Omega) $ to which $w_k$, up to subsequences, converges weakly in $W^{1,p}_0(\Omega)$. Now, as $h_{n}(|v_{k}|)f_{n} \le n^{2}$  by the  classical Stampacchia's argument (see for instance \cite[Lemma 2]{bmp}), one has that $w_k\le C$ almost everywhere in $\Omega$ where $C$ is independent of $k$, i.e.  $w\in L^\infty(\Omega)$.

		\medskip
		
		Now we need to show that $w=S(v)$; i.e. we need to pass to the limit with respect to $k$ in the following formulation
		\begin{equation}\label{bis}
			\int_{\Omega} a(x,w_{k}, \nabla w_{k}) \cdot\nabla \phi + \int_{\Omega}g_{n}(w_{k})|\nabla w_{k}|^{p} \phi = \int_{\Omega}h_{n}(|v_{k}|)f_{n} \phi, \ \ \forall \phi \in W^{1,p}_{0}(\Omega) \cap L^{\infty}(\Omega).
		\end{equation}
		For the right-hand of \eqref{bis} one can  apply the Lebesgue Theorem since $h_{n}(|v_{k}|)f_{n}\varphi\le n^2\varphi \in L^1(\Omega)$. Now observe that, once one proves that $w_k$ converges strongly to $w$ as $k\to\infty$ in $W^{1,p}_0(\Omega)$, one can safely  pass to the limit on  the left-hand of \eqref{bis}. Indeed it will be sufficient recall that $g_n$ is bounded and that $a$ satisfies \eqref{cara2}. 
		
		\medskip
		
		The proof of the strong convergence in $W^{1,p}_0(\Omega)$ of $w_k$ is also quite  classical under the above assumptions. Anyway, for the sake of completeness, we will sketch it.

		Let consider $\varphi_{\rho}(s):= s e^{\rho s^{2}}$ $(\rho>0)$ which satisfies \begin{equation}\label{fi ro}
			\eta\varphi_{\rho}'(s)-\mu|\varphi_{\rho}(s)| \ge \frac{\eta}{2}, ~~~ \forall s \in \mathbb{R},  ~~~ \forall \eta,\mu >0, ~~~ \forall \rho \ge \frac{\mu^{2}}{4\eta^{2}}.
		\end{equation}
		We fix $\phi=\varphi_{\rho}(u_{k})$ in $(\ref{bis})$ where $u_{k}:=w_{k}-w \in W^{1,p}_{0}(\Omega) \cap L^{\infty}(\Omega)$, then we have
		\begin{equation}\label{debole}
			\int_{\Omega}a(x, w_{k}, \nabla w_{k}) \cdot \nabla u_{k} \varphi_{\rho}'(u_{k})  + \int_{\Omega} g_{n}(w_{k})|\nabla w_{k}|^{p} \varphi_{\rho}(u_{k})= \int_{\Omega} h_{n}(|v_{k}|)f_{n}\varphi_{\rho}(u_{k}).
		\end{equation}
		By \eqref{cara1} we find
		$$-\int_{\Omega} g_{n}(w_{k}) |\nabla w_{k}|^{p} \varphi_{\rho}(u_{k}) \le n 
		\int_{\Omega} |\nabla w_{k}|^{p} |\varphi_{\rho}(u_{k})| \le \frac{n}{\alpha} \int_{\Omega} 
		a(x,w_{k}, \nabla w_{k}) \cdot \nabla w_{k} |\varphi_{\rho}(u_{k})|,$$
		which gives 
		\begin{equation}\label{maggiorazione}
			\begin{aligned}	
				-\int_{\Omega} g_{n}(w_{k}) |\nabla w_{k}|^{p} \varphi_{\rho}(u_{k}) &\le \frac{n}{\alpha} \int_{\Omega} 
				a(x,w_{k}, \nabla w_{k}) \cdot \nabla u_{k} |\varphi_{\rho}(u_{k})| 
				\\
				&+ \frac{n}{\alpha} \int_{\Omega} a(x, w_{k}, \nabla w_{k}) \cdot \nabla w |\varphi_{\rho}(u_{k})|.
			\end{aligned}	
		\end{equation}
		Observe now that, since $a(x, w_{k},\nabla w_{k})$ is bounded in $L^{p'}(\Omega)^{N}$ and since $\nabla w | \varphi_{\rho}(u_{k})|$ strongly converges to zero in $L^{p}(\Omega)^{N}$ as $k\to\infty$, one has
		\begin{equation}\label{limite p.45}
			\lim_{k \to \infty} \int_{\Omega} a(x, w_{k}, \nabla w_{k}) \cdot \nabla w |\varphi_{\rho}(u_{k})| =0. 
		\end{equation}
		Moreover, an application of the Lebesgue Theorem gives that
		\begin{equation}\label{limite2}
			\lim_{k \to \infty} \int_{\Omega} h_{n}(|v_{k}|)f_{n}\varphi_{\rho}(u_{k})=0.
		\end{equation}
		Gathering  together $(\ref{debole})$, $(\ref{maggiorazione})$, $(\ref{limite p.45})$,  and \eqref{limite2} we have
		$$\int_{\Omega}a(x, w_{k}, \nabla w_{k}) \cdot \nabla u_{k} \varphi_{\rho}'(u_{k}) \le  \frac{n}{\alpha} \int_{\Omega} 
		a(x,w_{k}, \nabla w_{k}) \cdot \nabla u_{k} |\varphi_{\rho}(u_{k})| + \omega_{k},   $$
		
		where  $\omega_{k}$ denotes an infinitesimal quantity as $k\to\infty$. 
		
		Thus, using \eqref{fi ro} fixing $\rho=\frac{n^2}{4\alpha^2}$, one has 
		\begin{equation*}\label{1}
			\int_{\Omega} a(x, w_{k}, \nabla w_{k}) \cdot \nabla u_{k}  \le  \omega_{k},
		\end{equation*} 
		which implies
		\begin{equation}\label{2}
			\int_{\Omega} \left(a(x, w_{k}, \nabla w_{k}) -a(x, w_{k}, \nabla w)\right) \cdot   \nabla (w_{k}-w)  \le -\int_{\Omega} a(x,w_{k}, \nabla w) \cdot  \nabla  u_k +\omega_{k}.
		\end{equation}
		It follows from \eqref{cara2} and from the fact that $w_k$ converges to $w$ in $L^p(\Omega)$ as $k\to \infty$ that $a(x, w_{k}, \nabla w)$ strongly converges in $L^{p'}(\Omega)^{N}$ to $a(x, w, \nabla w)$ as $k\to\infty$. Moreover, as $u_k$ weakly converges to $0$ in $W^{1,p}_0(\Omega)$ as $k\to\infty$, from \eqref{2} and \eqref{cara3} one has that
		$$\lim_{k \to \infty} \int_{\Omega} (a(x, w_{k}, \nabla w_{k}) - a(x, w_{k}, \nabla w)) \cdot \nabla (w_{k}-w)=0,$$  
		which allows to apply Lemma 5 of \cite{bmp} in order to deduce that $$w_{k}\to w ~~~ \text{ strongly in $W^{1,p}_{0}(\Omega)$}.$$ This is sufficient to conclude that $w=S(v)$, i.e. $S$ is continuous. 
		
		\medskip
		
		We finally apply the Schauder fixed point theorem to conclude that $S$ has a nonnegative fixed point $u_{n} \in W^{1,p}_{0}(\Omega) \cap L^{\infty}(\Omega)$ that is a solution to $(\ref{problema apprissimante})$.  This concludes the proof in the case in which \eqref{problema apprissimante} admits a unique solution. 
		
		\medskip
		In the general case one can reason by performing a standard  approximation argument. In fact,  thanks to Leray-Lions  theory one can find  bounded nonnegative solutions $u_{n,m}\in W^{1,p}_0(\Omega)$ of
		$$
		\begin{cases}
			\displaystyle	-\operatorname{div}(a(x,u_{n,m}, \nabla u_{n,m})) + g_{n}(u_{n,m})\frac{|\nabla u_{n,m}|^{p}}{\displaystyle 1+\frac{|\nabla u_{n,m}|^{p}}{m}} = h_{n}(u_{n,m})f_{n} & \text{in $\Omega$,} \\
			u_{n,m}=0 & \text{on $\partial \Omega$,}
		\end{cases}
		$$
		along with suitable a priori estimates which are uniform with respect to $m$ (see for instance \cite{bmp}),  and  pass to the limit  as $m\to\infty$ obtaining  a bounded nonnegative solution $u_n\in W^{1,p}_0(\Omega)$ of \eqref{problema troncato}. 
		
		\bk	
	\end{proof}

	\subsection{A priori estimates}
	In this section we collect all the estimates from which one could  derive the existence of a limit function for $u$ of $u_n$.

	\begin{lemma}\label{lemma esistenza u_p}  Let $a$ satisfy \eqref{cara1}-\eqref{cara3} with $1<p<N$, let $u_{n}$ be a nonnegative solution to problem \eqref{problema troncato}. Then $u_n$ is bounded in $W^{1,p}_0(\Omega)$ with respect to $n$. Moreover $g_{n}(u_{n})| \nabla u_{n}|^{p}$ is bounded in $L^1(\Omega)$ and $h_{n}(u_{n})f_{n}$ is bounded in $L^1_{\rm loc}(\Omega)$ with respect to $n$.	
	\end{lemma}
	
	\begin{proof}
		
		We start proving that $u_{n}$ is bounded in $W^{1,p}_{0}(\Omega)$.
		Let us first observe that, from \eqref{g infinito}, there exists $\overline{k}>0$ such that for all $s \in [\overline{k}, \infty)$ one has $g(s)\geq\eta>0$ for some $\eta>0$. Then we choose $T_{\overline{k}}(u_{n})$ as test function in the weak formulation of \eqref{problema troncato} in order to deduce that
		$$\int_{\Omega} a(x,u_{n},\nabla u_{n}) \cdot \nabla T_{\overline{k}}(u_{n}) + \int_{\Omega} g_{n}(u_{n}) | \nabla u_{n}|^{p} T_{\overline{k}}(u_{n}) = \int_{\Omega} h_{n}(u_{n}) f_{n} T_{\overline{k}}(u_{n}),$$ which implies that
		$$\begin{aligned}
			&\int_{\{u_{n} \le \overline{k}\}} a(x,u_{n}, \nabla u_{n}) \cdot \nabla u_{n} + \int_{\{u_{n} \le \overline{k}\}} g_{n}(u_{n})| \nabla u_{n}|^{p}u_{n} + \overline{k} \int_{\{u_{n}>\overline{k}\}} g_{n}(u_{n}) | \nabla u_{n}|^{p} \\ &\le \int_{\{u_{n} \le s_{1}\}} h_{n}(u_{n}) f_{n} u_{n} + \overline{k} \int_{\{u_{n} >s_{1} \}} h_{n}(u_{n}) f_{n}.
		\end{aligned} $$
		Without loosing generality let assume $n\ge \eta$ and, using that $g_{n}(u_{n})> \eta$ on $\{u_n>\overline{k}\}$, \eqref{cara1} and $(\ref{h vicino 0})$, one yields to 
		$$\alpha \int_{\{u_{n} \le \overline{k}\}} | \nabla u_{n}|^{p} + \overline{k} \eta \int_{\{u_{n} >\overline{k}\}} | \nabla u_{n}|^{p} \le c_{1}\int_{\{u_{n} \le s_{1}\}}  u_{n}^{1-\gamma} f_{n} + \overline{k}\sup_{s \in [s_{1}, \infty)} h(s) \int_{\{u_{n} >s_{1}\}} f_{n}.$$
		From the previous it follows
		$$\min\{\alpha,\overline{k}\eta\} \int_{\Omega} | \nabla u_{n}|^{p} \le \left(c_{1}s_{1}^{1-\gamma}+ \overline{k}\sup_{s \in [s_{1}, \infty)} h(s) \right) \|f\|_{L^{1}(\Omega)},$$
		namely $u_n$ is bounded in $W^{1,p}_0(\Omega)$ with respect to $n$ since $\gamma\le 1$ and thanks to \eqref{h infinito}.
		
		\medskip
		
		Now we focus on proving the $L^1$-estimate on $g_n(u_n)|\nabla u_n|^p$ in $n$. Let us take $T_{1}(u_{n})$ as test function in the weak formulation of \eqref{problema troncato}, obtaining 
		$$\int_{\Omega} a(x, u_{n}, \nabla u_{n}) \cdot \nabla T_{1}(u_{n}) + \int_{\Omega} g_{n}(u_{n}) | \nabla u_{n}|^{p} T_{1}(u_{n}) = \int_{\Omega}  h_{n}(u_{n}) f_{n} T_{1}(u_{n}),$$ 
		which, recalling \eqref{cara1} and \eqref{h vicino 0}, gives	
		\begin{equation}\label{stimag}
			\begin{aligned}
				\int_{\{u_{n} \ge 1\}} g_{n}(u_{n}) | \nabla u_{n}|^{p} &\le c_{1}\int_{\{u_{n} \le s_{1}\}} u_{n}^{1-\gamma} f_{n} + \sup_{s \in [s_{1}, \infty)}h(s) \int_{\{u_{n}>s_{1}\}} f_{n} 
				\\
				&\le \left(c_{1}s_{1}^{1-\gamma}+\sup_{s \in [s_{1}, \infty)}h(s)\right) \|f\|_{L^{1}(\Omega)}.
			\end{aligned}	
		\end{equation}
		Moreover one can observe that
		\begin{equation}\label{gn limitata uniformemente}
			\int_{\{u_{n} < 1 \}}  g_{n}(u_{n}) | \nabla u_{n}|^{p} \le \max_{s\in[0,1]}g(s) \int_\Omega |\nabla u_n|^p \le C,
		\end{equation}
		where $C$ does not depend on $n$ since $u_n$ is bounded in $W^{1,p}_0(\Omega)$ with respect to $n$. Then it follows from \eqref{stimag} and \eqref{gn limitata uniformemente} that $g_{n}(u_{n})| \nabla u_{n}|^{p}$ is bounded in $L^1(\Omega)$ with respect to $n$.

		\medskip
		
		We finally show that $h_{n}(u_{n})f_{n}$ is bounded in $L^1_{\rm loc}(\Omega)$ with respect to $n$.
		
		Let consider $0\le \varphi \in C^{1}_{c}(\Omega)$ as a test function in the weak formulation of \eqref{problema troncato}, obtaining
		\begin{equation*}
			\begin{aligned}
				\int_{\Omega}  h_{n}(u_{n})f_{n} \varphi = \int_{\Omega} a(x, u_{n}, \nabla u_{n}) \cdot\nabla \varphi + \int_{\Omega} g_{n}(u_{n}) |\nabla u_{n}|^p \varphi\le C, 
			\end{aligned}
		\end{equation*}
		where $C$ does not depend on $n$. Indeed we have already proven that $g_n(u_n)|\nabla u_{n}|^p$ is bounded in $L^1(\Omega)$. Moreover, as $u_n$ is bounded in $W^{1,p}_0(\Omega)$ and thanks to \eqref{cara2}, it is easy to check that $a(x,u_n,\nabla u_n)$ is bounded in $L^{p'}(\Omega)^N$ with respect to $n$.
	\end{proof}

	In the next lemma we show the existence of a limit function for $u_n$ with respect to $n$ and we show that any truncation of $u_n$ strongly converges in $W^{1,p}_0(\Omega)$ as $n\to\infty$.
	
	\begin{lemma}\label{lem_esistenzau_p>1}
		Let $a$ satisfy \eqref{cara1}-\eqref{cara3} with $1<p<N$, let $u_{n}$ be a nonnegative solution to problem \eqref{problema troncato}. Then there exists $u\in W^{1,p}_0(\Omega)$ to which $u_n$, up to subsequences, converges as $n\to\infty$ almost everywhere in $\Omega$. Moreover $T_k(u_n)$ converges, up to subsequences, strongly in $W^{1,p}_0(\Omega)$ to $T_k(u)$ as $n\to\infty$ for any $k>0$. Finally $g(u)|\nabla u|^p \in L^1(\Omega)$ and $h(u)f\in L^1_{\rm loc}(\Omega)$.
	\end{lemma}
	\begin{proof}
		The existence of a limit function $u$ follows from a standard compactness argument once that Lemma \ref{lemma esistenza u_p} is in force. From the same lemma one also has that $h_n(u_n)f_n$ is bounded in $L^1_{\rm loc}(\Omega)$ with respect to $n$; then an application of the Fatou Lemma as $n\to\infty$ gives that $h(u)f\in L^1_{\rm loc}(\Omega)$.
		
		\medskip
		
		To show the strong convergence of $T_k(u_n)$ in $n$, we re-adapt a classical  idea of \cite{bbm}.  
		
		\medskip
		
		We recall that the function $\varphi_{\rho}(s):= s e^{\rho s^{2}}$ $(\rho>0)$ satisfies \eqref{fi ro} and we define for any $k>0$
		$$w_{n,k}:=T_{k}(u_{n})-T_{k}(u)\,.$$ 
		
		We take $\varphi_{\rho}(w_{n,k})$ as a test function in the weak formulation of \eqref{problema troncato}; one has
		\begin{eqnarray}\label{convtron1}
			\into a(x, u_{n}, \nabla u_{n}) \cdot \nabla w_{n,k} \varphi_{\rho}'(w_{n,k}) + \int_{\Omega} g_{n}(u_{n}) |\nabla u_{n}|^{p} \varphi_{\rho}(w_{n,k}) = \int_{\Omega} h_{n}(u_{n}) f_{n} \varphi_{\rho}(w_{n,k}).
		\end{eqnarray}
		
		We can write the first term on the left-hand of the previous as  
		\begin{equation*}
			\begin{aligned}
				&\int_{\{u_{n}\le k\}} a(x,u_{n}, \nabla u_{n}) \cdot \nabla w_{n,k} \varphi_{\rho}'(w_{n,k}) - \int_{\{u_{n}> k\}} a(x,u_{n}, \nabla u_{n}) \cdot \nabla T_k(u) \varphi_{\rho}'(w_{n,k}) 
				\\
				&\ge \into a(x, T_{k}(u_{n}), \nabla T_{k}(u_{n})) \cdot \nabla w_{n,k} \varphi'_{\rho}(w_{n,k}) \\ & -\beta \int_{\{u_{n}>k\}}  \left( b(x)+|u_{n}|^{p-1} + |\nabla  u_{n}|^{p-1} \right) | \nabla T_{k}(u)| |\varphi_{\rho}'(w_{n,k})|\,, 
			\end{aligned}
		\end{equation*}
		where in the last step we used \eqref{cara2}. 
		Gathering the previous inequality into \eqref{convtron1} one has
		\begin{equation*}
			\begin{aligned}
				&\int_{\Omega} a(x, T_{k}(u_{n}), \nabla T_{k}(u_{n})) \cdot \nabla (T_{k}(u_{n})-T_{k}(u)) \varphi'_{\rho}(w_{n,k})+\int_{\Omega}g_{n}(u_{n}) |\nabla u_{n}|^{p} \varphi_{\rho}(w_{n,k}) \le \\ &\int_{\Omega}h_{n}(u_{n}) f_{n} \varphi_{\rho}(w_{n,k})+\beta \int_{\{u_{n}>k\}}  \left( b(x)+|u_{n}|^{p-1} + |\nabla  u_{n}|^{p-1} \right) | \nabla T_{k}(u)| |\varphi'_{\rho}(w_{n,k})|.
			\end{aligned}
		\end{equation*}
		Since $|\nabla T_{k}(u)| \chi_{\{u_{n}>k\}} \to 0$ strongly in $L^{p}(\Omega)$ as $n\to\infty$ while $\beta \left( b(x)+|u_{n}|^{p-1} + |\nabla  u_{n}|^{p-1} \right) |\varphi'_{\rho}(w_{n,k})|$ is bounded in $L^{p'}(\Omega)$ since $u_n$ is bounded in $W^{1,p}_0(\Omega)$ with respect to  $n$, then the last term in previous inequality tends to zero as $n$ tends to infinity.
		Hence, 
		\begin{equation}\label{convtron2}
			\begin{aligned}
				&\int_{\Omega} a(x, T_{k}(u_{n}), \nabla T_{k}(u_{n})) \cdot \nabla (T_{k}(u_{n})-T_{k}(u)) \varphi'_{\rho}(w_{n,k})
				\\
				&\le -\int_{\Omega} g_{n}(u_{n}) | \nabla u_{n}|^{p} \varphi_{\rho}(w_{n,k}) + \int_{\Omega}h_{n}(u_{n}) f_{n} \varphi_{\rho}(w_{n,k})+ \omega_n,	
			\end{aligned}
		\end{equation}
		where,  again, $\omega_n$ is a quantity that tends to $0$ as $n\to\infty$.
		
		Now observe that for the first term on the right-hand of \eqref{convtron2}, one has 
		\begin{eqnarray*}
			-\int_{\Omega} g_{n}(u_{n}) |\nabla u_{n}|^{p} \varphi_{\rho}(w_{n,k}) \le 
			\int_{\{u_{n}\le k\}} g_{n}(u_{n}) |\nabla u_{n}|^{p} | \varphi_{\rho}(w_{n,k})|,
		\end{eqnarray*}
		since $\varphi_{\rho}(w_{n,k}) \ge 0$ on $\{u_{n} > k\}$. 
		Furthermore, using \eqref{cara1}, one deduces 
		\begin{eqnarray*}
			-\int_{\Omega} g_{n}(u_{n}) |\nabla u_{n}|^{p} \varphi_{\rho}(w_{n,k}) \le \displaystyle \frac{\max_{s\in[0,k]}g(s)}{\alpha}\int_{\Omega} a(x, T_{k}(u_{n}), \nabla T_{k}(u_{n})) \cdot \nabla T_{k}(u_{n}) | \varphi_{\rho}(w_{n,k})|\,.
		\end{eqnarray*}
		Since $a(x,T_{k}(u_{n}), \nabla T_{k}(u_{n}))$ is bounded in $L^{p'}(\Omega)^{N}$ while $\nabla T_{k}(u) | \varphi_{\rho}(w_{n,k})|$ strongly converges to zero in $L^{p}(\Omega)^{N}$ as $n\to\infty$, one yields to 
		$$\lim_{n \to \infty} \int_{\Omega}a(x,T_{k}(u_{n}), \nabla T_{k}(u_{n})) \cdot \nabla T_{k}(u) | \varphi_{\rho}(w_{n,k})|=0,$$
		which implies 
		$$ 
		\begin{aligned}
			&-\int_{\Omega} g_{n}(u_{n}) |\nabla u_{n}|^{p} \varphi_{\rho}(w_{n,k}) 
			\\
			&\le\frac{\max_{s\in[0,k]}g(s)}{\alpha}\int_{\Omega} a(x, T_{k}(u_{n}), \nabla T_{k}(u_{n})) \cdot \nabla (T_{k}(u_{n})- T_{k}(u)) | \varphi_{\rho}(w_{n,k})|+\omega_n.	
		\end{aligned}
		$$
		Then gathering the previous into \eqref{convtron2} one has
		\begin{equation*}
			\begin{aligned}
				&\int_{\Omega}  a(x, T_{k}(u_{n}), \nabla T_{k}(u_{n})) \cdot \nabla (T_{k}(u_{n})-T_{k}(u)) \varphi_{\rho}'(w_{n,k}) \\ &\le \frac{\max_{s\in[0,k]}g(s)}{\alpha}\int_{\Omega} a(x, T_{k}(u_{n}), \nabla T_{k}(u_{n})) \cdot \nabla (T_{k}(u_{n})- T_{k}(u)) | \varphi_{\rho}(w_{n,k})|
				\\
				&+  \int_{\Omega}  h_{n}(u_{n}) f_{n} \varphi_{\rho}(w_{n,k}) + \omega_n,
			\end{aligned}
		\end{equation*}
		
		which, using $(\ref{fi ro})$ with $\rho=\left(\frac{\max_{s\in[0,k]}g(s)}{2\alpha}\right)^2$, implies that 
		\begin{equation}\label{convtron3}
			\begin{aligned}
				\int_{\Omega} a(x,T_{k}(u_{n}), \nabla T_{k}(u_{n})) \cdot \nabla(T_{k}(u_{n})-T_{k}(u)) \le 2\int_{\Omega}  h_{n}(u_{n})f_{n} \varphi_{\rho}(w_{n,k}) +\omega_n.
			\end{aligned}
		\end{equation}
		Now observe that it follows from \eqref{cara2} and from having $u_n$ weakly converging in $W^{1,p}_0(\Omega)$  that 
		$$\lim_{n \to \infty} \int_{\Omega} a(x,T_{k}(u_{n}), \nabla T_{k}(u)) \cdot \nabla(T_{k}(u_{n})-T_{k}(u))=0.$$
		Then by adding and subtracting this quantity  into \eqref{convtron3}, one has 
		\begin{equation*}
			\begin{aligned}
				&\int_{\Omega} (a(x,T_{k}(u_{n}), \nabla T_{k}(u_{n}))-a(x,T_{k}(u_{n}), \nabla T_{k}(u))) \cdot \nabla(T_{k}(u_{n})-T_{k}(u))  \\ &\le  2\int_{\Omega}  h_{n}(u_{n})f_{n} \varphi_{\rho}(w_{n,k})+\omega_n.
			\end{aligned} 
		\end{equation*}
		We claim that  the right-hand of the previous inequality  converges to $0$ as $n\to \infty$, for every fixed $k>0$. If $h(0)<\infty$ then this passage to the limit follows from the convergence in $L^1(\Omega)$ of $f_n$ coupled with the *-weak convergence of $\varphi_{\rho}(w_{n,k})$ to zero in $L^{\infty}(\Omega)$. 
		\medskip 
		
		Hence,  we assume $h(0)=\infty$.
		We fix $0<\delta < s_1$. Then, using \eqref{h vicino 0} one yields to 
		\begin{equation}\label{convtron4bis}
			\int_{\Omega}  h_{n}(u_{n})f_{n}\varphi_{\rho}(w_{n,k}) \le c_{1} \int_{\{u_{n} \le \delta\}}  \frac{f_{n}}{u_{n}^{\gamma}}\varphi_{\rho}(w_{n,k}) + \sup_{s \in [\delta, \infty)} h(s) \int_{\{u_{n}>\delta\}} f_{n} \varphi_{\rho}(w_{n,k}).
		\end{equation}
		The second term on the right-hand of \eqref{convtron4bis}  converges to $0$ as $n\to\infty$ since $f_n$ converges in $L^1(\Omega)$ while $\varphi_\rho(w_{n,k})$ converges *-weakly in $L^\infty(\Omega)$ to $0$ as $n\to\infty$. 
		For the first term of \eqref{convtron4bis} one reason as
		\begin{equation*}\label{convtron5}
			c_{1}\int_{\{u_{n} \le \delta\}}  \frac{f_{n}}{u_{n}^{\gamma}}\varphi_{\rho}(w_{n,k}) \le 2 c_{1}\int_{\{u_{n} \le \delta \}}  \delta^{1- \gamma}f ~ e^{\rho w_{n,k}^{2}}.
		\end{equation*}	
		
		Applying the Lebesgue Theorem  we can say that (here we may assume $\delta\le 1$)
		\begin{equation*}\label{convtron6}
			\lim_{n \to \infty} 2 c_{1} \int_{\{u_{n} \le \delta \}}  \delta^{1- \gamma}f ~ e^{\rho w_{n,k}^{2}} = 2 c_{1}\int_{\{u \le \delta\}}  \delta^{1-\gamma}f \le 2 c_{1}\int_{\{u \le \delta\}}  f,
		\end{equation*}
		since $\gamma\le 1$. We have already shown that $h(u)f \in L^{1}_{\rm{loc}}(\Omega)$. This implies that $\{u=0\} \subset \{ f=0\}$ up to a set of zero Lebesgue measure, which gives
		$$\lim_{\delta \to 0} \int_{\{u \le \delta\}} f = \int_{\{u=0\}} f = 0.$$
		
		This allows to deduce that for every $k>0$
		\begin{equation*}
			\lim_{n \to \infty} \int_{\Omega} (a(x,T_k(u_{n}), \nabla T_k(u_{n})) - a(x,T_k(u_{n}), \nabla T_k(u))) \cdot (\nabla T_{k}(u_{n})- \nabla T_{k}(u))=0,
		\end{equation*}
		which is sufficient to apply \cite[Lemma 5]{bmp}, deducing 
		$$T_{k}(u_{n}) \to T_{k}(u) ~~\text{strongly in $W^{1,p}_{0}(\Omega)$},$$
		for every $k>0$. This also implies that $\nabla u_n$ converges almost everywhere in $\Omega$ to $\nabla u$ as $n\to\infty$. Moreover, as Lemma \ref{lem_esistenzau_p>1} guarantees that $g_n(u_n)|\nabla u_n|^p$ is bounded in $L^1(\Omega)$ with respect to $n$, an application of the Fatou Lemma in $n$ gives that $g(u)|\nabla u|^p \in L^1(\Omega)$. This concludes the proof.
	\end{proof}
	
	\begin{remark}\label{convqo}
		In Lemma \ref{lem_esistenzau_p>1} we have shown $T_k(u_n)$ converges, up to subsequences, strongly in $W^{1,p}_0(\Omega)$ to $T_k(u)$ as $n\to\infty$ for any $k>0$. From this fact we have that $\nabla u_n$ converges almost everywhere in $\Omega$ to $\nabla u$ as $n\to\infty$. This can  be  used to deduce below that $a(x,u_{n}, \nabla u_{n})$ strongly converges to $a(x,u, \nabla u)$ in $L^{p'}(\Omega)^{N}$ as $n\to\infty$ thanks to \eqref{cara2}.
	\end{remark}

	\subsection{Proof of the existence result}
	
	We are now ready to prove the main existence result of this section, i.e. Theorem \ref{teorema p>1}.
	\begin{proof}[Proof of Theorem \ref{teorema p>1}]
		
		Let $u_n$ be a solution to \eqref{problema troncato} whose existence is guaranteed by Lemma \ref{lem_esistenzaun}. Moreover it follows from Lemma \ref{lem_esistenzau_p>1} that there exists $u\in W^{1,p}_{0}(\Omega)$ which is, up to subsequences, the almost everywhere limit in $\Omega$ of $u_n$ as $n\to\infty$. Moreover the same lemma gives that $g(u)|\nabla u|^p \in L^1(\Omega)$ and that $h(u)f \in L^{1}_{\rm{loc}}(\Omega)$.
		
		\medskip

		To conclude the proof, it remains to show that $u$ satisfies \eqref{def_p>1}, namely that one can pass to the limit with respect to $n$: 
		\begin{equation}\label{form_n}
			\int_{\Omega} a(x,u_{n}, \nabla u_{n}) \cdot \nabla \varphi + \int_{\Omega} g_{n}(u_{n}) |\nabla u_{n}|^{p} \varphi = \int_{\Omega}  h_{n}(u_{n}) f_{n} \varphi,
		\end{equation}
		where $\varphi \in C^{1}_{c}(\Omega)$.
		
		\medskip

		It is a consequence of both \eqref{cara2} and Lemma \ref{lemma esistenza u_p} that $a(x,u_{n}, \nabla u_{n})$ strongly converges to $a(x,u, \nabla u)$ in $L^{p'}(\Omega)^{N}$ as $n\to\infty$. This is sufficient to take $n\to\infty$ in the first term of \eqref{form_n} (see also Remark \ref{convqo}).

		\medskip

		For the second term we show the equi-integrability of the sequence $g_{n}(u_{n})| \nabla u_{n}|^{p}$.
		
		To this aim we introduce the following function:
		$$S_{\eta,k}(s):= \begin{cases}
			0 &\text{$s \le k$,}\\
			\frac{s-k}{\eta} & \text{$k<s<k+\eta$,}\\
			1 & \text{$s \ge k + \eta$,}
		\end{cases}$$
		where $k>0$ is a fixed parameter.
		Let us take $S_{\eta,k}(u_{n})$ with $k>0$ as a test function in the weak formulation of \eqref{problema troncato}, yielding to
		$$\int_{\Omega} a(x,u_{n}, \nabla u_{n}) \cdot \nabla u_{n} S_{\eta, k}'(u_{n}) + \int_{\Omega} g_{n}(u_{n}) | \nabla u_{n}|^{p} S_{\eta, k}(u_{n}) = \int_{\Omega}  h_{n}(u_{n}) f_{n
		} S_{\eta, k}(u_{n}),$$
		which, thanks to \eqref{cara1}, implies
		\begin{equation*}
			\int_{\Omega} g_{n}(u_{n}) | \nabla u_{n}|^{p} S_{\eta,k}(u_{n})\le \sup_{s \in [k, \infty)} h(s) \int_{\{u_{n} > k\}} f.
		\end{equation*}
		We take $\eta \to 0$ applying the Fatou Lemma, obtaining 
		\begin{equation*}\label{piccolezza}
			\int_{\{u_{n} > k\}} g_{n}(u_{n}) |\nabla u_{n}|^{p} \le \sup_{s \in [k,  \infty)} h(s) \int_{\{u_{n} > k\}} f.
		\end{equation*}
		Recalling that $T_k(u_n)$ strongly converges in $W^{1,p}_0(\Omega)$ with respect to $n$ and that $g$ is continuous, the former inequality gives the equi-integrability of the sequence $g_{n}(u_{n})| \nabla u_{n}|^{p}$ with respect to $n$. This fact, jointly with 
		$$g_n(u_n)|\nabla u_n|^p \to g(u)|\nabla u|^p \ \ \text{a.e. in } \ \Omega \ \text{as}\ \ \ n\to\infty,$$
		and recalling that  $g(u)|\nabla u|^p\in L^1(\Omega)$, allows to apply the Vitali Theorem to deduce that 
		$$g_{n}(u_{n})| \nabla u_{n}|^{p}\to g(u)|\nabla u|^p \ \ \text{strongly converges in}\ \  L^1(\Omega) \ \text{as}\ \ \ n\to\infty.$$
		
		\medskip

		The proof is then concluded once we pass to the limit in the right-hand side of \eqref{form_n}  
		for any $0\le \varphi \in C^1_c(\Omega)$, as the case of $\varphi$ with general sign will easily follow. If $h(0)<\infty$ the passage to the limit is trivial. Hence, without loss of generality, we assume that $h(0)=\infty$.
		
		\medskip
		
		We choose $\delta > 0$ such that $\delta \not \in \{\eta : |\{u=\eta\}|>0\}$ and we split \eqref{form_n} as follows 
		\begin{equation}\label{convhp>10}
			\int_{\Omega} a(x,u_{n}, \nabla u_{n}) \cdot \nabla \varphi + \int_{\Omega} g_{n}(u_{n}) |\nabla u_{n}|^{p} \varphi =  \int_{\{u_{n} \le \delta\}} h_{n}(u_{n}) f_{n} \varphi + \int_{\{u_{n}> \delta \}} h_{n}(u_{n}) f_{n} \varphi.
		\end{equation}
		We first show that 
		\begin{equation}\label{limite su un piccolo}
			\lim_{\delta \to 0} \limsup_{n \to \infty} \int_{\{u_{n} \le \delta \}} h_{n}(u_{n}) f_{n} \varphi=0.
		\end{equation}
		Let fix $V_{\delta}(u_{n}) \varphi$ ($V_\delta$ is defined in \eqref{Vdelta}) with $0 \le \varphi \in C^{1}_{c}(\Omega)$ as test function in the weak formulation of \eqref{problema troncato}, and we deduce
		$$\int_{\Omega} a(x, u_{n}, \nabla u_{n}) \cdot \nabla \left( V_{\delta}(u_{n}) \varphi \right)+ \int_{\Omega} g_{n}(u_{n}) | \nabla u_{n}|^{p} V_{\delta}(u_{n}) \varphi = \int_{\Omega} h_{n}(u_{n}) f_{n} V_{\delta}(u_{n}) \varphi. $$
		Then one has
		\begin{equation*}
			\int_{\{u_{n} \le \delta\}} h_{n}(u_{n}) f_{n} \varphi \le \int_{\Omega} h_{n}(u_{n}) f_{n}  V_{\delta}(u_{n}) \varphi \le \int_{\Omega} a(x, u_{n}, \nabla u_{n}) \cdot \nabla \varphi V_{\delta}(u_{n}) +  \int_{\Omega} g_{n}(u_{n}) | \nabla u_{n}|^{p} V_{\delta}(u_{n}) \varphi.
		\end{equation*}

		Now we can take the limsup as $n\to\infty$ in the previous inequality; indeed for the first term on the right-hand one recalls Remark \ref{convqo} and the fact that $V_\delta\le 1$. For the second term on the right-hand we have already proven that  $g_{n}(u_{n})| \nabla u_{n}|^{p}$ strongly converges in $L^1(\Omega)$ to $g(u)| \nabla u|^{p}$ as $n\to\infty$. Then one has
		
		\begin{equation}\label{convhp>1}
			\limsup_{n \to \infty} \int_{\{u_{n} \le \delta\}} h_{n}(u_{n}) f_{n} \varphi \le \int_{\Omega} a(x, u, \nabla u) \cdot \nabla \varphi V_{\delta}(u) + \int_{\Omega} g(u) | \nabla u|^{p} V_{\delta }(u) \varphi. 
		\end{equation} 
		Now let $\delta \to 0$ in \eqref{convhp>1} applying the Lebesgue Theorem, deducing that
		$$\lim_{\delta \to 0} \limsup_{n \to \infty} \int_{\{u_{n} \le \delta\}} h_{n}(u_{n}) f_{n} \varphi \le \int_{\{u=0\}} a(x,u,\nabla u) \cdot \nabla \varphi + \int_{\{u=0\}} g(u) | \nabla u|^{p} \varphi =0,$$
		since $u \in W^{1,p}_{0}(\Omega)$ and $a(x,0,0)=0$ for all $x \in \Omega$ (recall $\nabla u=0$ a.e. on $\{u=0\}$). 
		
		\medskip
		
		Now we focus on the second term on the right-hand of  \eqref{convhp>10}. Observing that 
		$$h_{n}(u_{n}) f_{n} \varphi \chi_{\{u_{n}>\delta\}} \le \sup_{s\in[\delta,\infty)}h(s) f\varphi \in L^1(\Omega)$$
		then one can take $n\to\infty$, yielding to
		$$\lim_{n \to \infty} \int_{\{u_{n}> \delta \}} h_{n}(u_{n}) f_{n} \varphi = \int_{\{u > \delta \}} h(u)f\varphi,$$
		since $|\{u=\delta\}|=0$.
		
		Moreover, as $h(u)f\in L^1_{\rm loc}(\Omega)$, one can take $\delta \to 0$
		\begin{equation}\label{convhp>12}
			\lim_{\delta \to 0} \lim_{n \to \infty} \int_{\{u_{n}> \delta \}} h_{n}(u_{n}) f_{n} \varphi  = \int_{\{u>0\}} h(u)f \varphi = \int_{\Omega} h(u)f \varphi,
		\end{equation}
		where the last equality follows from the fact $\{u=0\} \subset \{f=0\}$ up to a set of zero Lebesgue measure since $h(u)f$ is locally integrable. 
		
		Finally observe that \eqref{limite su un piccolo} and \eqref{convhp>12} allow us to pass to the limit as $n\to \infty$, for fixed $\delta>0$,  and then as $\delta\to 0$ in  \eqref{convhp>10} and the proof is concluded. 
		
	\end{proof}

	\begin{remark}\label{estensione funzioni test}
		In Theorem \ref{teorema p>1} we found a distributional solution $u\in W^{1,p}_0(\Omega)$ satisfying $g(u)|\nabla u|^p\in L^1(\Omega)$. Then it is worth mentioning that in this case we can extend the class of test function given in \eqref{def_p>1} to $W^{1,p}_{0}(\Omega) \cap L^{\infty}(\Omega)$. 
		
		\medskip

		Indeed, for any $0\le v \in W^{1,p}_{0}(\Omega) \cap L^{\infty}(\Omega)$, there exists $\varphi_{n} \in C^{1}_{c}(\Omega)$ such that $\varphi_{n} \to v$ strongly in $W^{1,p}_{0}(\Omega)$ as $n\to\infty$. Moreover let $\rho_{\eta}$ be a standard mollifier. We note that $\psi_{n,\eta}=\rho_{\eta}* \min\{v, \varphi_n\} \in C^{1}_{c}(\Omega)$ for $\eta>0$ small enough, hence it is an admissible test function for the problem $(\ref{problema})$, so we can write
		\begin{equation}\label{rem_ex}
			\int_{\Omega} a(x, u, \nabla u) \cdot \nabla \psi_{n,\eta}+ \int_{\Omega} g(u) |\nabla u|^{p}  \psi_{n,\eta}= \int_{\Omega} h(u) f \psi_{n,\eta}.
		\end{equation}
		We recall that $\psi_{n,\eta} \to \psi_n=\min\{\varphi_{n},v\}$ as $\eta \to 0$ strongly in $W^{1,p}(\Omega)$ and *-weak in $L^{\infty}(\Omega)$. Then, as $a(x,u,\nabla u)\in L^{p'}(\Omega)^N$, $g(u)|\nabla u|^p \in L^1(\Omega)$ and $h(u)f\in L^1_{\rm loc}(\Omega)$, one can  take $\eta \to 0$ in \eqref{rem_ex} obtaining 
		\begin{equation}\label{rem_ex2}
			\int_{\Omega} a(x, u, \nabla u) \cdot \nabla \psi_{n}+ \int_{\Omega} g(u) |\nabla u|^{p}  \psi_{n}= \int_{\Omega} h(u) f \psi_{n}.
		\end{equation}

		Now note that $\psi_{n} \to v$ strongly in $W^{1,p}(\Omega)$ and *-weak in $L^{\infty}(\Omega)$ as $n\to\infty$, so we can  take $n\to\infty$ in the first two terms of \eqref{rem_ex2}.
		
		For the term on the right-hand of \eqref{rem_ex2} one can reason as follows. Firstly observe that an application of the Fatou Lemma with respect to $n$ gives that
		$$\int_{\Omega} h(u) f v \le \liminf_{n \to \infty} \int_{\Omega} h(u)f \psi_{n} = \int_{\Omega} a(x, u, \nabla u) \cdot \nabla v + \int_{\Omega} g(u) |\nabla u|^{p} v,$$
		whose right-hand is finite. Then, as $h(u)fv \in L^{1}(\Omega)$, one can apply the Lebesgue Theorem to obtain
		$$\lim_{n \to \infty} \int_{\Omega}  h(u)f \psi_{n} = \int_{\Omega} h(u)fv.$$
		Therefore, as the case of a function  $v$ with generic  sign easily follows,  we have proven that the solution $u$ to \eqref{problema} found in Theorem \ref{teorema p>1} satisfies  
		$$\int_{\Omega} a(x, u, \nabla u)\cdot \nabla v + \int_{\Omega} g(u) |\nabla u|^{p}  v = \int_{\Omega} h(u) f v,$$  for all $v \in W^{1,p}_{0}(\Omega) \cap L^{\infty}(\Omega)$.
	\end{remark}

	\section{Main assumptions and existence result for $p=1$}
	\label{sec:p=1}
	
	In this section we address  the limit case $p=1$. In particular we are interested in proving existence of nonnegative solutions to the following Dirichlet boundary value problem 
	\begin{equation}\label{pb12}
		\begin{cases}
			- \Delta_1 u +g(u)|Du|=h(u)f &\text{in $\Omega$,} \\
			u=0 &\text{on $\partial \Omega$,}
		\end{cases}
	\end{equation}
	where  $f$ is a positive function in $L^{1}(\Omega)$, $g:[0, \infty) \to [0, \infty)$ is a {positive,} bounded {and} continuous function such that \eqref{g infinito} is in force. The function $h:[0,  \infty) \to [0,\infty]$ is  continuous and possibly singular  with $h(0) \neq 0$,  it  is finite outside the origin and such that \eqref{h vicino 0} and \eqref{h infinito} hold.
	
	\medskip
	Here is how the notion of solution for problem \eqref{pb12} is intended.

	\begin{defin}\label{def_p=1}
		Let $0<f \in L^{1}(\Omega)$. A nonnegative  $u \in BV(\Omega)$ is a solution of problem \eqref{pb12} if $D^{j}u=0$, $g(u) \in L^{1}_{\rm{loc}}(\Omega, |Du|)$ and $h(u)f \in L^{1}_{\rm{loc}}(\Omega)$ and if there exists a vector field $z \in \mathcal{DM}^{\infty}_{\rm{loc}}(\Omega)$, with $\|z\|_{L^{\infty}(\Omega)^{N}} \le 1$ satisfying
		\begin{equation}\label{p=1 in senso distribuzionale}
			-\operatorname{div}z+g(u)|Du|=h(u)f \ \ \ \text{in} \ \mathcal{D}'(\Omega),
		\end{equation}
		\begin{equation}\label{parring di anzellotti}
			(z,DT_k(u))=|DT_k(u)| ~~~ \text{as measures in $\Omega$ for any $k>0$,}
		\end{equation}
		
		and
		
		\begin{equation}\label{condizione al bordo per p=1}
			u(x)=0 ~~~\text{for $\mathcal{H}^{N-1}$-a.e. $x \in \partial \Omega$.}
		\end{equation}
	\end{defin}
	\begin{remark}\label{423}
		Let us spend a few words on Definition \ref{def_p=1}. Formula \eqref{parring di anzellotti} is the weak way in which the vector field $z$ plays the role of the singular quotient $Du|Du|^{-1}$. Hence the  \eqref{p=1 in senso distribuzionale} and \eqref{parring di anzellotti}  represent   the weak  way the first term  in \eqref{pb12} is intended.  
		
		Furthermore, the boundary datum is given by \eqref{condizione al bordo per p=1} which is something strongly related to the presence of $g$. Indeed, it is classical nowadays that solutions to $1$-Laplace Dirichlet problems do not attain, in general, the boundary datum pointwise when  $g\equiv 0$, and  in this  case, a weaker condition involving $[z,\nu]$ is usually required (see for instance \cite{ABCM, ACM, dgop}).	
		
		Let us finally 	explicitly observe that  if  $h(0)=\infty$,   as $h(u)f \in L^{1}_{\rm{loc}}(\Omega)$,  then, again, $\{u=0\}\subset \{f=0\}$. So that in this case, as $f>0$,   then $u>0$.  
	\end{remark}
	
	We define the following function which will be widely used in the sequel
	$$\Gamma_{p}(s):=\int_{0}^{s} g^{\frac{1}{p}}(\sigma) \ d \sigma.$$
	Moreover, we denote by $\Gamma(s):=\Gamma_1(s)$. Let explicitly observe that, as $g$ is bounded, $\Gamma(s)$ is a Lipschitz function satisfying assumptions of Lemma \ref{chainrule}. In Section \ref{sec:ex_ginf} we briefly discuss the case of a $g$ possibly unbounded at infinity. 
	
	\medskip
	
	A very similar reasoning to the one of \cite[Remark 3.4]{ms} gives the following result.
	
	\begin{proposition}
		Let  $u$ be a  solution of the problem $(\ref{pb12})$ in the sense of Definition \ref{def_p=1}, then \begin{equation}\label{equazione per p=1 con esponenzionale}
			-\operatorname{div}(z e^{-\Gamma(u)})=  h(u) fe^{-\Gamma(u)}  \ \  \text{in} \ \mathcal{D}'(\Omega).
		\end{equation}
	\end{proposition}
	\begin{proof}
		By $(\ref{parring di anzellotti})$	and  \cite[Proposition 3.3]{ls1} 
		we have that 
		$$\theta(z, D(1- e^{-\Gamma(u)}),x)=1 \ \ \ \  \text{$|D(1-e^{-\Gamma(u)})|-$a.e. in $\Omega$,} $$
		where $\theta(z, D v,x)$ is the Radon–Nikod\'ym derivative of $(z,D v)$ with respect to $|D v|$ provided $v\in BV (\Omega)$. 
		Consequently, for all Borel sets $B \subset \Omega,$
		$$\int_{B}(z,D(1- e^{-\Gamma(u)}))=\int_{B} \theta(z, D(1- e^{-\Gamma(u)}),x)~ |D(1-e^{-\Gamma(u)})|= \int_{B} |D(1-e^{-\Gamma(u)})|. $$
		Therefore \begin{equation}\label{misure con esponenziale}
			(z,D(1- e^{-\Gamma(u)}))= |D(1-e^{-\Gamma(u)})| ~~~ \text{as Radon measures in $\Omega$. }
		\end{equation}
		On the other hand, by \eqref{27}, $(\ref{misure con esponenziale})$, $(\ref{p=1 in senso distribuzionale})$ and Lemma \ref{chainrule}, we have 
		\begin{equation*}
			\begin{aligned}
				-&\operatorname{div}( e^{- \Gamma(u)}z)=\operatorname{div}((1- e^{-\Gamma(u)})z)-\operatorname{div}z =(z, D(1-e^{-\Gamma(u)}))+(1- e^{-\Gamma(u)}) \operatorname{div}z-\operatorname{div}z= \\ =&|D(1- e^{-\Gamma(u)})| - ( e^{-\Gamma(u)}) \operatorname{div}z=( e^{-\Gamma(u)})g(u)|Du|-( e^{-\Gamma(u)})(-h(u)f+g(u)|Du|)= e^{-\Gamma(u)}h(u)f,
			\end{aligned}  
		\end{equation*}
		i.e. we obtain $(\ref{equazione per p=1 con esponenzionale})$.
	\end{proof}
	
	Let us then state the main  result of this section.
	
	\begin{theorem}\label{teorema per p=1}
		Let $g$ be positive, bounded and satisfying \eqref{g infinito} and let $h$ satisfy \eqref{h vicino 0} and \eqref{h infinito}. Finally let $0<f \in \textit{L}^{1}(\Omega)$. Then there exists a solution to \eqref{pb12} in the sense of Definition \ref{def_p=1}.
	\end{theorem}
	
	\subsection{Approximation scheme and existence of a limit function}
	
	The proof of Theorem \ref{teorema per p=1} will be presented as an application of a series of lemmas. We introduce the following approximation scheme: 
	\begin{equation}\label{approxp=1}
		\begin{cases}
			- \Delta_p u_p + g(u_p)|\nabla u_p|^p=h(u_p)f &\text{in $\Omega$,} \\
			u_p=0 &\text{on $\partial \Omega$,}
		\end{cases}
	\end{equation}
	whose existence of $u_p\in W^{1,p}_0(\Omega)$ in the sense of Definition \ref{soluzione p} has been proven in Theorem \ref{teorema p>1}. Let us explicitly observe that, as long as we deal with the solution found in the mentioned theorem, Remark \ref{estensione funzioni test} is in force; this means that the set of test functions is enlarged to $W^{1,p}_0(\Omega)\cap L^\infty(\Omega)$.
	
	\medskip

	We firstly look for some uniform estimates on $u_p$ for $p>1$ small enough. Without loss of generality and for the sake of exposition, by {\it uniformly bounded with respect to $p$} we mean the existence of $p_0>1$ with some estimate holding for any $1<p\leq p_0$. 
	
	\begin{lemma}\label{lemma convergenza delle up ad u}
		Let $g$ be positive, bounded and satisfying \eqref{g infinito}, let $h$ satisfy \eqref{h vicino 0} and \eqref{h infinito} and let $0<f \in \textit{L}^{1}(\Omega)$. Let $u_{p}$ be the solution to \eqref{approxp=1} obtained in Theorem \ref{teorema p>1}. Then $u_{p}$ and $\Gamma_p(u_p)$ are uniformly bounded with respect to $p$ in $BV(\Omega)$. Moreover there exists $u \in BV(\Omega)$ such that $u_p$ converges to u (up to a subsequence) in $L^q(\Omega)$ for every $q<\frac{N}{N-1}$ and $\nabla u_p$ converges to $Du$ *-weakly as measures. Finally,  $\Gamma (u) \in BV(\Omega)$ and $h(u)f \in L^{1}_{\rm{loc}}(\Omega)$.
	\end{lemma}
	\begin{proof}
		We  observe  again that, as \eqref{g infinito} is in force, there exists $\overline{k}>0$ such that for all $s \in [\overline{k}, \infty)$ one has $g(s)\geq\eta>0$ for some $\eta>0$.
		
		We choose $T_{\overline{k}}(u_{p}) \in W^{1,p}_{0}(\Omega) \cap L^{\infty}(\Omega)$ as test function in the weak formulation of \eqref{approxp=1} (recall Remark $\ref{estensione funzioni test}$), so that $$\int_{\Omega} |\nabla u_{p}|^{p-2} \nabla u_{p} \cdot \nabla T_{\overline{k}}(u_{p})+ \int_{\Omega}g(u_{p})|\nabla u_{p}|^{p} T_{\overline{k}}(u_{p})  = \int_{\Omega} h(u_{p})fT_{\overline{k}}(u_{p}),$$
		from which
		\begin{equation}\label{stima1}
			\min\{\alpha, \overline{k}\eta\}\int_{\Omega} |\nabla u_{p}|^{p} \le \left(c_{1}s_{1}^{1-\gamma}+\sup_{s \in [s_{1}, \infty)}h(s) \overline{k} \right) \|f\|_{L^{1}(\Omega)}.
		\end{equation}	
		It follows from \eqref{stima1} and from an application of the Young inequality that 
		$$\int_{\Omega} |\nabla u_{p}| \le \frac{1}{p}\int_{\Omega} | \nabla u_{p}|^{p}+ \frac{1}{p'}|\Omega| \le \frac{1}{p \min\{\alpha, \overline{k}\eta\}}\left(c_{1}s_{1}^{1-\gamma}+\sup_{s \in [s_{1}, \infty[}h(s) \overline{k} \right) \|f\|_{L^{1}(\Omega)} + \frac{1}{p'}|\Omega|,$$
		which shows that $u_p$ is bounded in $BV(\Omega)$ with respect to $p$ since the right-hand of the previous is bounded with respect to $p$ and $u_p$ has zero trace on $\partial\Omega$.
		
		\medskip
		
		Then standard compactness result for $BV$ functions  assures that there exists $u \in BV(\Omega)$ such that, up to a subsequence, $u_p$ converges in $L^{q}(\Omega)$ for any $q<\frac{N}{N-1}$, almost everywhere in $\Omega$ and $\nabla u_{p}$ converges to $Du$ *-weakly as measures as $p\to 1^+$.
		
		\medskip
		
		To show that  $\Gamma(u_p)$ is bounded in $BV(\Omega)$ it is sufficient to reason  as for \eqref{stimag} and \eqref{gn limitata uniformemente} in order to deduce 
		\begin{equation}\label{stima secondo termine per p=1}
			\int_{\Omega} g(u_{p}) |\nabla u_{p}|^{p}  \le C,
		\end{equation} 
		where $C$ is a constant independent of  $p$, and to use the Young inequality to get
		\begin{equation}\label{240}
			\int_{\Omega}  |\nabla \Gamma_p(u_{p})|=	\int_{\Omega} g(u_{p})^{\frac{1}{p}} |\nabla u_{p}| \leq 	\int_{\Omega} g(u_{p}) |\nabla u_{p}|^{p}  +|\Omega| \le C+|\Omega|; 
		\end{equation}

		observe that, in particular, by weak lower semicontinuity one has that $\Gamma (u) \in BV(\Omega)$.

		\medskip
		
		Now we focus on showing that $h(u)f$ is locally integrable. We choose $0 \le \varphi \in C^{1}_{c}(\Omega)$ as a test function in the weak formulation of \eqref{approxp=1}; this yields to 
		\begin{equation}\label{stimah}
			\begin{aligned}
				\int_{\Omega} h(u_{p}) f \varphi &= \int_{\Omega}|\nabla u_{p}|^{p-2} \nabla u_{p} \cdot\nabla \varphi + \int_{\Omega} g(u_{p})|\nabla u_{p}|^{p} \varphi 
				\\
				&\le \frac{1}{p'} \int_{\Omega} |\nabla u_{p}|^{p} + \frac{1}{p} \int_{\Omega} | \nabla \varphi|^{p} +\int_{\Omega} g(u_{p})|\nabla u_{p}|^{p} \varphi.
			\end{aligned}	
		\end{equation}
		As the right-hand of \eqref{stimah} is bounded with respect to $p$ thanks to \eqref{stima1} and \eqref{stima secondo termine per p=1}, one has that $h(u_p)f$ is locally bounded in $L^1(\Omega)$ with respect to $p$.
		
		Finally an application of the Fatou Lemma as $p\to 1^+$ gives that $h(u)f \in L^1_{\rm loc}(\Omega)$. The proof is concluded.
	\end{proof}
	
	In the next Lemma we find a vector field $z$ which is the weak limit of $|\nabla u_p|^{p-2}\nabla u_p$ as $p\to 1^+$. 
	
	\begin{lemma}\label{lemesistenzaz}
		Let $g$ be positive, bounded and satisfying \eqref{g infinito}, let $h$ satisfy \eqref{h vicino 0} and \eqref{h infinito} and  let $0<f \in \textit{L}^{1}(\Omega)$. Moreover, let $u$ be the function found in Lemma \ref{lemma convergenza delle up ad u}. Then there exists a vector field {$z \in \DM_{\rm loc}(\Omega)$} with $\|z\|_{L^{\infty}(\Omega)^{N}} \le 1$,  such that 
		\begin{equation}\label{ze}
			z e^{-\Gamma(u)}\in \mathcal{DM}^{\infty}(\Omega),
		\end{equation}
		\begin{equation}\label{uguaglianza esponente gamma}
			-\operatorname{div}(z e^{-\Gamma(u)})=h(u)fe^{-\Gamma(u)} ~~~\text{in $\mathcal{D}'(\Omega)$,}
		\end{equation}
		and \begin{equation}\label{disuguaglianza in distribuzione dell'equazione}
			-\operatorname{div}z+|D \Gamma(u)| \le  h(u)f ~~~ \text{in $\mathcal{D}'(\Omega)$.}
		\end{equation}

	\end{lemma}
	\begin{proof}
		Let $u_{p}$ be the solution to \eqref{approxp=1} obtained in Theorem \ref{teorema p>1}. Then it follows from \eqref{stima1} and from an application of the H\"older inequality that, for $1 \le q < p'$, it holds
		\begin{equation}\label{stima campo di anzellotti}
			\int_{\Omega} \left| |\nabla u_{p}|^{p-2} \nabla u_{p}\right|^{q}= \int_{\Omega}|\nabla u_{p}|^{q(p-1)} \le \left( \int_{\Omega} |\nabla u_{p}|^{p} \right)^{\frac{q(p-1)}{p}}|\Omega|^{1-\frac{q(p-1)}{p}} \le C^{\frac{q(p-1)}{p}} |\Omega|^{1-\frac{q(p-1)}{p}}.
		\end{equation}
		Hence $|\nabla u_{p}|^{p-2} \nabla u_{p}$ is bounded in $L^{q}(\Omega)^N$ with respect to $p$ and there exists $z_q\in L^{q}(\Omega)^N$ such that $|\nabla u_{p}|^{p-2} \nabla u_{p} \rightharpoonup z_{q}$ weakly in $L^{q}(\Omega)^N$, for all $q<\infty$.
		Moreover it follows from the lower semicontinuity in $(\ref{stima campo di anzellotti})$ with respect to $p$ that $$\|z\|_{L^q(\Omega)^N} \le |\Omega|^{\frac{1}{q}}, ~~~ \forall q < \infty,$$
		and thus letting $q \to \infty$ then $z \in L^{\infty}(\Omega)^{N}$ with $\|z\|_{L^\infty(\Omega)^N} \le 1$.

		\medskip
		
		Let us now show \eqref{uguaglianza esponente gamma}; let us take $e^{- \Gamma(u_{p})} \varphi$ as a test function in the weak formulation of \eqref{approxp=1} where $0\le \varphi \in C^{1}_{c}(\Omega)$, yielding to
		$$\int_{\Omega} |\nabla u_{p}|^{p-2} \nabla u_{p} \cdot \nabla \varphi e^{- \Gamma(u_{p})}= \int_{\Omega} h(u_{p})f e^{- \Gamma(u_{p})} \varphi.$$
		We can pass to the limit in  the left-hand of the previous since $e^{-\Gamma(u_p)}|\nabla u_p|^{p-2}\nabla u_p$ converges to $e^{-\Gamma(u)}z$ in $L^q(\Omega)^{N}$ for any $q<\infty$ as $p\to 1^+$. This shows that 
		\begin{equation*}\label{covergenza lato destro campo ed esponenziale}
			\lim_{p \to 1^+} \int_{\Omega} | \nabla u_{p}|^{p-2}\nabla u_{p} \cdot \nabla \varphi e^{- \Gamma(u_{p})}= \int_{\Omega}z \cdot \nabla \varphi ~e^{- \Gamma(u)}.
		\end{equation*}
		For the right-hand we distinguish two cases: if h is finite at the origin then the passage to the limit is trivial. Hence, without losing generality we assume that $h(0)=\infty$. We first split the integral as
		\begin{equation}\label{passaggiolimiteh}
			\int_{\Omega} h(u_{p})f e^{-\Gamma(u_{p})} \varphi = \int_{\{u_{p} \le \delta\}} h(u_{p})f e^{-\Gamma(u_{p})} \varphi + \int_{\{u_{p}>\delta\}} h(u_{p})f e^{-\Gamma(u_{p})} \varphi,
		\end{equation}
		where $\delta \not\in \{\eta : |\{u=\eta\}|>0\}$ which is at most a countable set.
		
		Observe that it follows from Lemma \ref{lemma convergenza delle up ad u} that $h(u)f$ is locally integrable. Since $h(0)=\infty$ and $f>0$, then $u>0$ almost everywhere in $\Omega$.
		Moreover, since
		$$\chi_{\{u_{p}>\delta\}}h(u_{p})fe^{-\Gamma(u_p)}\varphi \le \sup_{s\in [\delta,\infty)}h(s) f\varphi \in L^1(\Omega),$$
		and
		$$\chi_{\{u>\delta\}}h(u)fe^{-\Gamma(u)}\varphi \le h(u)f\varphi \in L^1(\Omega),$$ 
		one can apply twice the Lebesgue Theorem, deducing that
		\begin{equation}\label{hf esponenzionale limite in p e delta}
			\lim_{\delta \to 0}\lim_{p \to 1^+} \int_{\{u_{p}>\delta\}}h(u_{p})f e^{-\Gamma(u_{p})} \varphi = \lim_{\delta \to 0} \int_{\{u> \delta\}} h(u)f e^{- \Gamma(u)} \varphi \stackrel{u>0}{=} \int_{\Omega} h(u)f e^{- \Gamma(u)} \varphi.
		\end{equation}

		Now we analyze the first term on the right-hand of \eqref{passaggiolimiteh}, we fix $V_{\delta}(u_{p}) e^{- \Gamma(u_{p})} \varphi$ ($V_\delta$ is defined in \eqref{Vdelta}) with $0 \le \varphi \in C^{1}_{c}(\Omega)$ as test function in the weak formulation of \eqref{approxp=1}, obtaining 
		\begin{equation*}
			\begin{aligned}
				&\int_{\Omega}|\nabla u_{p}|^{p}  V'_{\delta}(u_{p})  e^{- \Gamma(u_{p})} \varphi - \int_{\Omega} g(u_{p})| \nabla u_{p}|^{p}  V_{\delta}(u_{p})  e^{-\Gamma(u_{p})}   \\ +&\int_{\Omega} | \nabla u_{p}|^{p-2} \nabla u_{p} \cdot  \nabla \varphi V_{\delta}(u_{p}) e^{- \Gamma(u_{p})}  +\int_{\Omega} g(u_{p}) | \nabla u_{p}|^{p} V_{\delta}(u_{p}) e^{ -\Gamma(u_{p})}  \varphi= \int_{\Omega} h(u_{p})f V_{\delta}(u_{p}) e^{-\Gamma(u_{p})}  \varphi,
			\end{aligned}
		\end{equation*}
		which implies
		$$\int_{\{u_{p}\le \delta\}} h(u_{p})f e^{- \Gamma(u_{p})}\varphi \le \int_{\Omega} h(u_{p}) f V_{\delta}(u_{p}) e^{- \Gamma(u_{p})} \varphi \le  \int_{\Omega} |\nabla u_{p}|^{p-2} \nabla u_{p} \cdot \nabla \varphi V_{\delta}(u_{p}) e^{- \Gamma(u_{p})}. $$
		Through the Lebesgue Theorem we deduce
		\begin{equation}\label{hf esponenzionale delta minore zero}
			\lim_{\delta \to 0}\limsup_{p \to 1^+} \int_{\{u_{p} \le \delta\}}  h(u_{p}) f e^{-\Gamma(u_{p})}\varphi \le \int_{\{u=0\}} z \cdot \nabla \varphi e^{-\Gamma(u)} \stackrel{u>0}{=} 0.
		\end{equation}
		From \eqref{hf esponenzionale limite in p e delta} and \eqref{hf esponenzionale delta minore zero}, one gets 
		\begin{equation}\label{hf esponenziale finale}
			\lim_{p \to 1^+} \int_{\Omega}  h(u_{p}) f e^{-\Gamma(u_{p})} \varphi = \int_{\Omega}  h(u) f e^{-\Gamma(u)} \varphi.
		\end{equation}
		Hence we have shown \eqref{uguaglianza esponente gamma}.
		
		\medskip
		
		Observe that  as  $h(u)f e^{- \Gamma(u)} \geq 0 $ then, by  \cite[Lemma 2.3]{gop},   $z e^{- \Gamma(u)} \in \DM (\Omega)$, namely \eqref{ze}. 
		
		\medskip
		Finally we show \eqref{disuguaglianza in distribuzione dell'equazione}. Recalling \eqref{240} one has that, up to  subsequences, $\Gamma_{p}(u_{p}) \to \Gamma(u)$ a.e. as $p\to 1^+$, and 
		$$\int_{\Omega} |D \Gamma(u)| \varphi \le \liminf_{p \to 1^+} \int_{\Omega} | \nabla \Gamma_{p}(u_{p})|^{p} \varphi \le C,$$
		by weak lower semicontinuity (again we also used Young's inequality as for \eqref{240}). Observe that $C$ does not depend on $p$ thanks to \eqref{stima1}.
		
		Hence this allows to take $p\to 1^+$ in \eqref{problema} obtaining
		\begin{equation}\label{zdm}
			\int_{\Omega} z \cdot \nabla \varphi + \int_{\Omega} | D \Gamma (u)| \varphi \le \int_{\Omega}  h(u)f \varphi,
		\end{equation} 
		where for the right-hand we have reasoned analogously as to proven \eqref{hf esponenziale finale}.  Indeed, one has that ($\delta\le 1$) 
		\begin{equation*}\label{passaggioh}
			\lim_{\delta \to 0}\limsup_{p \to 1^+} \int_{\{u_{p} \le \delta\}}  h(u_{p}) f \varphi \le \lim_{\delta \to 0}\limsup_{p \to 1^+} \frac{1}{e^{-\Gamma(1)}}\int_{\{u_{p} \le \delta\}}  h(u_{p}) f e^{-\Gamma(u_{p})} \varphi =0.
		\end{equation*}
		
		{Let also note that \eqref{zdm} gives $z\in \DM_{\rm loc}(\Omega)$.} 	This concludes the proof.
	\end{proof}

	\subsection{Identification of the vector field $z$ and boundary datum}
	
	For the next two lemmas we need to define the following function
	\begin{equation}\label{tildeGamma}
		\tilde{\Gamma}_p(s) := \int_0^s (T_k(\sigma)g(\sigma))^\frac{1
		}{p} \ d\sigma,	
	\end{equation}
	where $\tilde{\Gamma}(s):=\tilde{\Gamma}_1(s)$.
	
	The next result  clarifies the role of $z$.
	\begin{lemma}\label{lemidentificazionez}
		Let $g$ be positive, bounded and satisfying \eqref{g infinito}, let $h$ satisfy \eqref{h vicino 0} and \eqref{h infinito} and  let $0<f \in \textit{L}^{1}(\Omega)$. Moreover, let $u$ be the function found in Lemma \ref{lemma convergenza delle up ad u} and let $z$ be the vector field found in Lemma \ref{lemesistenzaz}. It holds both
		\begin{equation}\label{equazionep=1}
			-\operatorname{div}z+g(u) |Du|=h(u)f ~~~ \text{in $\mathcal{D}'(\Omega)$,}
		\end{equation}
		and
		\begin{equation}\label{peru}
			{- T_k(u) \operatorname{div}z + T_k(u) |D\Gamma(u)| = h(u)fT_k(u) \ \ 	\text{in } \mathcal{D}'(\Omega) \text{ for any $k>0$.}}
		\end{equation} 		
		Finally it also holds 
		\begin{equation}\label{dju}
			D^j u =0,
		\end{equation}
		and
		\begin{equation}\label{pairing}
			\left(z,DT_k(u) \right)=|DT_k(u)| \ \ \text{as measures in }\Omega \text{ for any $k>0$}.
		\end{equation}	
	\end{lemma}
	\begin{proof}
		We highlight that here we need to make use of the new pairing introduced in \eqref{dist1d} applied to $\beta(s)=-e^{-s}$ and $v=\Gamma(u)$.

		\medskip

		One has 
		\begin{equation}\label{422}
			\begin{aligned}
				(e^{-\Gamma(u)})^{\#}|D \Gamma (u)|&\stackrel{\eqref{disuguaglianza in distribuzione dell'equazione}}{\le} (e^{-\Gamma(u)})h(u)f + (e^{-\Gamma(u)})^{\#}\operatorname{div}z 	\stackrel{\eqref{uguaglianza esponente gamma}}{=}-\operatorname{div}(e^{-\Gamma(u)}z)+ (e^{-\Gamma(u)})^{\#} \operatorname{div}z
				\\&\stackrel{\eqref{dist1d}}{=}\left(z, D [(- e^{-\Gamma(u)})^{\#}] \right) \stackrel{\eqref{finitetotal1d}}{\le} |D (-e^{-\Gamma(u)})|\stackrel{ \eqref{cr} }{=}(e^{-\Gamma(u)})^{\#}|D \Gamma (u)|\,,
			\end{aligned}
		\end{equation}
		which implies that all the inequalities in \eqref{422} are actually equalities.

		\medskip 
		In particular, 
		$$
		\left(z, D [(- e^{-\Gamma(u)})^{\#}] \right) = |D (-e^{-\Gamma(u)})|\,.
		$$ 
		Recalling \eqref{disuguaglianza in distribuzione dell'equazione}, we observe that one can apply Lemma \ref{lemma derivata salto nulla} with $\alpha(s)=s$, $\beta (s)=-e^{-s}$,  and $w=\Gamma(u)$ in order to deduce that $|D^{j}\Gamma(u)|=0$ and, as $\Gamma$ is increasing,   that \eqref{dju} holds.

		\medskip 
	Now we want to show the reverse inequality of \eqref{disuguaglianza in distribuzione dell'equazione} in order to get \eqref{equazionep=1}; one has 
		\begin{equation}\label{423e}
			\begin{aligned}
				e^{-\Gamma(u)}|D \Gamma (u)|&\stackrel{\eqref{disuguaglianza in distribuzione dell'equazione}}{\le} e^{-\Gamma(u)}(h(u)f + \operatorname{div}z) 	\stackrel{\eqref{uguaglianza esponente gamma}}{=} -\operatorname{div}(e^{-\Gamma(u)}z)+ (e^{-\Gamma(u)})\operatorname{div}z \\ & 
				\stackrel{\eqref{dist1}}{=}\left(z, D \left(- e^{-\Gamma(u)}\right) \right)  \stackrel{\eqref{finitetotal1}}{\le} |D (-e^{-\Gamma(u)})| \stackrel{ \eqref{cr} }{\le}e^{-\Gamma(u)}|D \Gamma (u)|.
			\end{aligned}
		\end{equation}
		Hence, we get a chain of equalities between measures that in particular implies that
		$$ e^{-\Gamma(u)}(-\operatorname{div}z+|D\Gamma(u)|)= e^{-\Gamma(u)} h(u)f ~~\text{in $\mathcal{D}'(\Omega)$.} $$
		
		Since  $\Gamma(u) \in BV(\Omega)$ (see  Lemma \ref{lemma convergenza delle up ad u}) it is finite $\mathcal{H}^{N-1}$-a.e., then   $e^{-\Gamma(u)}>0$ $\mathcal{H}^{N-1}\text{-a.e. in }  \Omega$; thus,  we get the reverse inequality of \eqref{disuguaglianza in distribuzione dell'equazione}, that is 
		\begin{equation}\label{bju}
			-\operatorname{div}z+|D \Gamma(u)| =  h(u)f ~~~ \text{in} \ \mathcal{D}'(\Omega),
		\end{equation}
		which, in turn, by applying Lemma \ref{chainrule} gives \eqref{equazionep=1}. 
		
		\medskip

		 	To prove \eqref{peru} we test  \eqref{bju} with $(\rho_\epsilon*T_k(u))\varphi$ where $k>0$, $0\le \varphi \in C^1_c(\Omega)$ and $\rho_\epsilon$ is a sequence of smooth mollifiers. For $\varepsilon$ small enough, this takes to 		
			$$-\int_\Omega (\rho_\epsilon*T_k(u))\varphi \operatorname{div}z + \int_{\Omega} (\rho_\epsilon*T_k(u))\varphi |D\Gamma(u)| = \int_{\Omega}h(u)f(\rho_\epsilon*T_k(u))\varphi.$$
			Now observe that $(\rho_\epsilon*T_k(u))$ converges $\mathcal{H}^{N-1}$ a.e. to $T_k(u)^*$ as $\epsilon\to 0$ and $T_k(u)^*\le k$. 
			Then, as it follows from Lemma \ref{lemma convergenza delle up ad u} that $h(u)f \in L^1_{\rm_{loc}}(\Omega)$ and $\Gamma(u) \in BV(\Omega)$, one can take $\epsilon\to 0$ applying the Lebesgue Theorem. This implies that \eqref{peru} holds.
		
		\medskip

		It is left to show \eqref{pairing}; we take $T_k(u_p)\varphi$ with $0\le\varphi\in C^1_c(\Omega)$ as a test function in the weak formulation of \eqref{approxp=1}; this takes to 
		\begin{equation}\label{provacampoz1}
			\int_\Omega |\nabla T_k(u_p)|^p\varphi + \int_\Omega |\nabla u_p|^{p-2}\nabla u_p \cdot \nabla \varphi T_k(u_p) +\int_\Omega |\nabla \tilde{\Gamma}_p(u_p)|^p\varphi = \int_\Omega h(u_p)fT_k(u_p)\varphi,
		\end{equation}
		where $\tilde{\Gamma}_p$ is defined in \eqref{tildeGamma}. Now observe that an application of the Young inequality gives
		\begin{equation}\label{provacampoz2}
			\int_\Omega |\nabla T_k(u_p)|\varphi + \int_\Omega |\nabla \tilde{\Gamma}_p(u_p)|\varphi \le \frac{1}{p}\int_\Omega |\nabla T_k(u_p)|^p\varphi + \frac{1}{p}\int_\Omega |\nabla \tilde{\Gamma}(u_p)|^p\varphi + \frac{2(p-1)}{p}\int_\Omega \varphi.
		\end{equation}
		Hence gathering \eqref{provacampoz2} into \eqref{provacampoz1} yields to
		\begin{equation}\label{provacampoz3}
			\begin{aligned}
				&\int_\Omega |\nabla T_k(u_p)|\varphi + \int_\Omega |\nabla \tilde{\Gamma}_p(u_p)|\varphi 
				\\
				&\le \int_\Omega h(u_p)fT_k(u_p)\varphi - \int_\Omega |\nabla u_p|^{p-2}\nabla u_p \cdot \nabla \varphi T_k(u_p) + \frac{2(p-1)}{p}\int_\Omega \varphi.
			\end{aligned}
		\end{equation}
		Hence we can take the liminf as $p\to 1^+$ in \eqref{provacampoz3} using weak lower semicontinuity for the left-hand. For the first term on the right-hand one can use the Lebesgue Theorem since ($\delta <s_1< k$) 
		$$h(u_p)fT_k(u_p)\le c_1 \delta^{1-\gamma}f \chi_{\{u_p\le \delta\}} + k\sup_{s\in[\delta,\infty)}h(s) f \chi_{\{u_p> \delta\}},$$
		which is strongly compact in $L^1(\Omega)$ with respect to $p$. The second term on the right-hand  passes to the limit while the third term degenerates as $p\to 1^+$. Hence one has
		\begin{equation*}
			\begin{aligned}
				\int_\Omega |D T_k(u)|\varphi + \int_\Omega |D \tilde{\Gamma}(u)|\varphi &\le \int_\Omega h(u)fT_k(u)\varphi - \int_\Omega z \cdot \nabla \varphi T_k(u) 
				\\
				&\stackrel{\eqref{peru}}{=} -\int_\Omega T_k(u)\operatorname{div}z \varphi + \int_\Omega |D\tilde{\Gamma}(u)|\varphi - \int_\Omega z \cdot \nabla \varphi T_k(u)
				\\
				&= \int_\Omega (z, DT_k(u))\varphi + \int_\Omega |D\tilde{\Gamma}(u)|\varphi,
			\end{aligned}
		\end{equation*}
		{where we also got advantage of $D^ju =0$, by writing $T_k(u)|D\Gamma(u)| = |D\tilde{\Gamma}(u)|$ }
		In particular this means 
		\begin{equation*}
			\int_\Omega |D T_k(u)|\varphi \le  \int_\Omega (z, DT_k(u))\varphi,
		\end{equation*}
		and, being the reverse inequality trivial, this shows \eqref{pairing}. 
		
		The proof is concluded.
	\end{proof}
	
	\begin{remark}
		
		We explicitly  remark that the request of positivity on $g$  is  needed to deduce that $\Gamma$ is an increasing function. It is worth mentioning that Theorem \ref{teorema per p=1} continues to hold in case  $g$ is only nonnegative but $\Gamma$ is  still a well defined increasing function.
	\end{remark}
	
	Finally we deal with the boundary datum.
	
	\begin{lemma}\label{lemdatobordo}
		Let $g$ be positive, bounded and satisfying \eqref{g infinito}, let $h$ satisfy \eqref{h vicino 0} and \eqref{h infinito} and  let $0<f \in \textit{L}^{1}(\Omega)$. Moreover, let $u$ be the function found in Lemma \ref{lemma convergenza delle up ad u} and let $z$ be the vector field found in Lemma \ref{lemesistenzaz}. Then $u(x)=0$ $\mathcal{H}^{N-1}$ almost everywhere in $\Omega$.
	\end{lemma}
	\begin{proof}
		Let $u_{p}$ be the solution to \eqref{approxp=1} obtained in Theorem \ref{teorema p>1}. Then let us take $T_{k}(u_{p})$ as test function in \eqref{approxp=1}, obtaining
		$$ \int_{\Omega} | \nabla T_k(u_p)|^{p} + \int_\Omega T_k(u_p)g(u_p)|\nabla u_{p}|^p = \int_{\Omega}h(u_{p})f T_{k}(u_{p}).$$
		For the right-hand once again one  observes that  ($k>\delta$)
		$$h(u_p)fT_k(u_p)\le \delta^{1-\gamma}f \chi_{\{u_p\le \delta\}} + k\sup_{s\in[\delta,\infty)}h(s) f \chi_{\{u_p> \delta\}},$$
		which is strongly compact in $L^1(\Omega)$ with respect to $p$. This allows to apply the generalized Lebesgue Theorem for the right-hand. 	Hence one can take the liminf as $p\to1^+$ in the previous; indeed one can use weak lower semicontinuity on the left-hand after an application of the Young inequality, recalling also that $u_p$ is zero on the boundary of $\Omega$. 
		
		This proves that
		\begin{equation}\label{datobordo1}
			\int_{\Omega} |D T_{k}(u)|+ \int_{\partial\Omega} T_k(u) \ d 	\mathcal{H}^{N-1} + \int_{\Omega} |D \tilde{\Gamma}(u)|+ \int_{\partial\Omega} \tilde{\Gamma}(u) \ d \mathcal{H}^{N-1} \le \int_{\Omega}h(u)f T_{k}(u),
		\end{equation}
		where $\tilde{\Gamma}$ is defined in \eqref{tildeGamma}.
		
		Since $h(u)fT_k(u) \in L^1(\Omega)$ and $\Gamma(u)\in BV(\Omega)$, one has
		\begin{equation}\label{datobordo2}
			\begin{aligned}
				\int_\Omega h(u)fT_k(u)&\stackrel{\eqref{peru}}{=} -\int_\Omega T_k(u)\operatorname{div}z + \int_\Omega |D\tilde{\Gamma}(u)| \\
				&\stackrel{\eqref{green}}{=} \int_\Omega (z,DT_k(u)) - \int_{\partial \Omega} [T_k(u)z, \nu] \ d\mathcal H^{N-1} + \int_\Omega |D\tilde{\Gamma}(u)|				\\
				&= \int_\Omega |DT_k(u)| - \int_{\partial \Omega} [T_k(u)z, \nu] \ d\mathcal H^{N-1} + \int_\Omega |D\tilde{\Gamma}(u)|.
			\end{aligned}
		\end{equation}
		Then gathering \eqref{datobordo2} into \eqref{datobordo1}, one yields to		
		\begin{equation*}
			\int_{\partial\Omega} ([T_k(u)z,\nu]+T_k(u)) \ d 	\mathcal{H}^{N-1} + \int_{\partial\Omega} \tilde{\Gamma}(u) \ d \mathcal{H}^{N-1} =0,
		\end{equation*}
		which, since $|[T_k(u)z,\nu]| \le T_k(u)$ on $\partial\Omega$ (recall \eqref{des3}), it gives that $\tilde{\Gamma}(u)$ (and so $u$) is identically null on $\partial\Omega$. This concludes the proof.	
	\end{proof}
	
	As consequence of the previous results we can now prove Theorem $\ref{teorema per p=1}$.
	\begin{proof}[Proof of Theorem $\ref{teorema per p=1}$]
		Let $u_{p}$ be the solution to \eqref{approxp=1} obtained in Theorem \ref{teorema p>1}. Then the proof follows from Lemmas \ref{lemma convergenza delle up ad u}, \ref{lemesistenzaz}, \ref{lemidentificazionez} and \ref{lemdatobordo}.
	\end{proof}
	
	\section{Some  extensions and remarks}
	\label{sec:ex}
	
	\subsection{The case of a nonnegative datum $f$}
	
	Up to now  we have required the positivity of the datum $f$. Now we want to consider  the case of a datum $f$ which is only  nonnegative; i.e. we consider  \begin{equation}\label{pb123}
		\begin{cases}
			- \Delta_1 u +g(u)|Du|=h(u)f &\text{in $\Omega$,} \\
			u=0 &\text{on $\partial \Omega$,}
		\end{cases}
	\end{equation}
	with $f\in L^1 (\Omega)$ being a nonnegative function, and $h,g$ as before.  
	\medskip
	
	Let explicitly underline that the extension of Theorem \ref{teorema per p=1}  is straightforward in the case  $h(0)<\infty$ and $f$  nonnegative the proof being de facto   the one already presented. Hence, without loosing generality, here  we assume $h(0)=\infty$.
	
	\medskip
	
	As we will see the existence of a solution can be obtained with some technical modifications both in the definition of solution of \eqref{pb123}, which  needs to be properly intended, and in the proof that is a suitable adaptation of the one of Theorem \ref{teorema per p=1}. 
	
	\medskip
	
	Here is  how the notion of  solution to \eqref{pb123}	has to be intended.
	
	\begin{defin}
		\label{weakdef_non}
		A nonnegative $u\in BV(\Omega)$ is a solution to \eqref{pb123} if $\chi_{\{u>0\}}\in BV_{\rm loc}(\Omega)$, $D^{j}u=0$, $g(u)\in L^1_{\rm loc}(\Omega, |Du|),  h(u)f \in L^1_{\rm loc}(\Omega)$,  and if there exists $ z\in \mathcal{D}\mathcal{M}^\infty_{\rm  loc}(\Omega)$ with $||z||_{L^\infty(\Omega)^N}\le 1$ such that 
		\begin{align}
			&-\chi^{\ast}_{\{u>0\}} \operatorname{div}z + g(u)|Du| =  h(u)f \ \ \text{in} \ \mathcal{D}'(\Omega), \label{def_distrp=1_non}
			\\
			&(z,DT_k(u))=|DT_k(u)| \label{def_zp=1_non} \ \ \ \ 	\text{as measures in } \ \Omega \text{ for any }k>0,
			\\
			& 	u(x)=0 \label{def_bordop=1_non} \ \ \ \text{for $\mathcal{H}^{N-1}$-a.e. $x \in \partial \Omega$.}
		\end{align}
	\end{defin} 
	
	\begin{remark}
		As one can see the main difference with respect to Definition \ref{def_p=1} consists in the presence of the characteristic function  $\chi_{\{u>0\}}$ in \eqref{def_distrp=1_non} which is a natural request as, in this case, we cannot infer $u>0$ from $h(u)f \in L^1_{\rm loc}(\Omega)$  as in Remark \ref{423}
	\end{remark}
	
	Let us state the existence result for this section.
	
	\begin{theorem}\label{teo_p=1_non}
		Let $g$ be positive, bounded and satisfying \eqref{g infinito} and let $h$ satisfy \eqref{h vicino 0} and \eqref{h infinito} with $h(0)=\infty$. Finally let $0\le f \in \textit{L}^{1}(\Omega)$. Then there exists a solution to \eqref{pb123} in the sense of Definition \ref{weakdef_non}.
	\end{theorem}
	\begin{proof}[Sketch of the proof] Here we only highlight  the authentic differences  with the proof of Theorem \ref{teorema per p=1}.
		We consider $u_p\in W^{1,p}_0(\Omega)$, solution to the approximating problems   in \eqref{approxp=1} and whose existence is proven in Theorem \ref{teorema p>1}.
		
		\medskip
		
		First observe that Lemma \ref{lemma convergenza delle up ad u} continues to hold in this case. Therefore,  there exists a nonnegative limit function $u\in BV(\Omega)$ for $u_p$, as $p\to 1^+$, such that $g(u)\in L^1(\Omega, |Du|)$ and  $h(u)f \in L^1_{\rm loc}(\Omega)$. Moreover, reasoning as in the first part of the proof of Lemma \ref{lemesistenzaz}, one gains the existence of a bounded vector field $z$ such that  $|z|\le 1$ in $\Omega$ with $|\nabla u_{p}|^{p-2} \nabla u_{p} \rightharpoonup z$ weakly in $L^{q}(\Omega)^N$ for all $q<\infty$. Let also underline that, as $-\Delta_p u_p = h(u_p)f - g(u_p)|\nabla u_p|^p$ is bounded with respect to  $p$ as measures, so that one deduces that $z\in \DM_{\rm{loc}}(\Omega)$.
		
		\medskip
		
		Now we focus on showing both $D^j u =0$ and \eqref{def_distrp=1_non}.
		
		\medskip
		
		Let $0\le\varphi\in C^1_c(\Omega)$ and let us take $(1-V_\delta(u_p))\varphi$ ($V_\delta$ is defined in \eqref{Vdelta}) as a test function in the weak formulation of \eqref{approxp=1}. Then, taking the liminf as $p\to 1^+$, one gains $\chi_{\{u>0\}}\in BV_{\rm loc}(\Omega)$, yielding to
		
		\begin{equation}\label{nonlimit1}
			-\Div(z\chi_{\{u>0\}}) + |D\chi_{\{u>0\}}| + \chi^*_{\{u>0\}}|D\Gamma(u)| \le h(u)f\chi_{\{u>0\}},
		\end{equation}
		in $\mathcal{D}'(\Omega)$. Moreover, as $|z|\le 1$ in $\Omega$, one has
		\begin{equation*}\label{nonlimit2}
			-\Div(z\chi_{\{u>0\}}) + |D\chi_{\{u>0\}}|\ge -\Div(z\chi_{\{u>0\}}) + (z,D\chi_{\{u>0\}}) \stackrel{\eqref{27}}{=} -\chi^*_{\{u>0\}}\Div z,
		\end{equation*}
		which, gathered into \eqref{nonlimit1}, gives 
		\begin{equation}\label{nonlimit3v}
			-\chi^*_{\{u>0\}}\Div z + \chi^*_{\{u>0\}}|D\Gamma(u)| \le h(u)f\chi_{\{u>0\}},
		\end{equation}
		in $\mathcal{D}'(\Omega)$. Let us stress that  \eqref{nonlimit1} gives that $z\chi_{\{u>0\}}\in \DM_{\rm{loc}}(\Omega)$ (recall \eqref{27}).
		
		Now let $0\le\varphi\in C^1_c(\Omega)$ and let us take $e^{-\Gamma(u_p)}\varphi$ as a test function in the weak formulation of \eqref{approxp=1}. Then, taking the liminf as $p\to1^+$ applying the Fatou Lemma, it allows to deduce
		
		\begin{equation}\label{nonlimit4}
			-\Div(ze^{-\Gamma(u)}) \ge h(u)fe^{-\Gamma(u)}.
		\end{equation}

		One  has
		
		\begin{equation*}
			\begin{aligned}
				(e^{-\Gamma(u)})^{\#}\chi^*_{\{u>0\}} |D \Gamma (u)|&\stackrel{\eqref{nonlimit3v}}{\le} e^{-\Gamma(u)}h(u)f\chi_{\{u>0\}} + (e^{-\Gamma(u)})^{\#}\chi^*_{\{u>0\}}\Div z
				\\
				&\stackrel{\eqref{nonlimit4}}{\le} -\Div(ze^{-\Gamma(u)})\chi^{*}_{\{u>0\}} +  (e^{-\Gamma(u)})^{\#}\chi^*_{\{u>0\}}\Div z 
				\\
				&\stackrel{\eqref{dist1d}}{=} \chi^{ *}_{\{u>0\}}\left(z, D (-e^{-\Gamma(u)})^{\#} \right)  \stackrel{\eqref{finitetotal1d}}{\le}\chi^*_{\{u>0\}} |D e^{-\Gamma(u)}|= (e^{-\Gamma(u)})^{\#} \chi^*_{\{u>0\}} |D \Gamma(u)|,
			\end{aligned}
		\end{equation*}
		where in the last equality we used Lemma \ref{chainrule}. This proves that
		\begin{equation}\label{pairingchi}
			\chi^{ *}_{\{u>0\}}\left(z, D (-e^{-\Gamma(u)})^{\#} \right) =\chi^*_{\{u>0\}} |D e^{-\Gamma(u)}|.
		\end{equation}
		
		Now, as in the proof of Lemma \ref{lemidentificazionez}, we want  to apply  an easy variation of  Lemma \ref{lemma derivata salto nulla}. In fact, as $\chi^{*}_{\{u>0\}}>0$ $\mathcal{H}^{N-1}$-a.e.  on $J_{\Gamma(u)}$, then the very same proof is still valid, and then,   using both \eqref{nonlimit3v} and \eqref{pairingchi} one deduces that  $D^j u = 0$,  and  that  \eqref{nonlimit3v} is equivalent to 
		\begin{equation}\label{nonlimit3}
			-\chi^*_{\{u>0\}}\Div z + |D\Gamma(u)| \le h(u)f\chi_{\{u>0\}}.
		\end{equation}
		
		From now on 	 the proof follows step-by-step the proof of Lemma \ref{lemidentificazionez}; in particular, a suited version of \eqref{423e} involving $\chi^*_{\{u>0\}}$ holds  allowing us to prove the validity of \eqref{def_distrp=1_non}. 
		Also as for  \eqref{peru} one readily gets 
		$$
		{- T_k(u) \chi^{\ast}_{\{u>0\}}\operatorname{div}z + T_k(u) |D\Gamma(u)| = h(u)fT_k(u) \ \ 	\text{in } \mathcal{D}'(\Omega) \text{ for any $k>0$.}}
		$$
		In particular,  this means that $T_k(u)\in L^1(\Omega, \Div z)$. 
		
		\medskip
		
		The proof of \eqref{def_zp=1_non} is not affected by the sign of the datum and follows as   for  \eqref{pairing}.

		\medskip
		
		Finally one has that Lemma \ref{lemdatobordo} applies without any modification in the proof. Indeed, as $T_k(u)\in L^1(\Omega, \Div z)$, one uses that $T_k(u)z\in \DM(\Omega)$ from Lemma \ref{poiu}. This shows \eqref{def_bordop=1_non}. The proof is  concluded.

	\end{proof}

	\subsection{The case of a nonnegative $g$ possibly blowing-up at infinity}
	
	\label{sec:ex_ginf}
	In Section \ref {sec:p=1} we only dealt with positive bounded functions $g$; this has been necessary to apply Lemma \ref{chainrule}, i.e. to deduce
	$$|D\Gamma (u)| = g(u)|Du|.$$
	By the way one can suitably modify  Definition \ref{def_p=1} in order to gain the  existence of a solution in a weaker sense.
	Let explicitly fix the notion of solution in which, due to what  we just said,  do not need to ask for $D^{j}u=0$.
	
	\begin{defin}\label{def_p=1_g_ill}
		Let $0<f \in L^{1}(\Omega)$. A nonnegative $u \in BV(\Omega)$ is a solution to \eqref{pb123} if  $\Gamma(u)\in BV_{\rm loc}(\Omega)$, $h(u)f \in L^{1}_{\rm{loc}}(\Omega)$, and if there exists a vector field $z \in \mathcal{DM}^{\infty}_{\rm{loc}}(\Omega)$, with $\|z\|_{L^\infty(\Omega)^N} \le 1$ satisfying
		\begin{equation}\label{p=1 in senso distribuzionale_g_ill}
			-\operatorname{div}z+|D\Gamma(u)|=h(u)f ~~~  \text{in $\mathcal{D}'(\Omega)$,}
		\end{equation}
		\begin{equation}\label{parring di anzellotti_g_ill}
			(z,DT_k(u))=|DT_k(u)| ~~~ \text{as measures in $\Omega$ for any $k>0$,}
		\end{equation}
		and
		\begin{equation}\label{condizione al bordo per p=1_g_ill}
			u(x)=0 ~~~\text{for $\mathcal{H}^{N-1}$-a.e. $x \in \partial \Omega$.}
		\end{equation}
	\end{defin}
	
	In case of a nonnegative $g$ possibly blowing-up at infinity we then have the following result. 
	
	\begin{theorem}\label{teorema per p=1_g_ill}
		Let $g$ satisfy \eqref{g infinito} and let $h$ satisfy \eqref{h vicino 0}-\eqref{h infinito}. Finally let $0<f \in \textit{L}^{1}(\Omega)$. Then there exists a solution $u \in BV(\Omega)$ to problem \eqref{pb123} in the sense of Definition \ref{def_p=1_g_ill}. Moreover, if $g$ is positive, then $D^{j}u=0$. 
	\end{theorem}
	\begin{proof}
		The proof is identical to the one of Theorem \ref{teorema per p=1} apart from the application of Lemma \ref{chainrule} which  is not needed in order to get \eqref{p=1 in senso distribuzionale_g_ill}.
	\end{proof}

	\section*{Acknowledgement}
	F. Oliva and F. Petitta are partially supported by the Gruppo Nazionale per l’Analisi Matematica, la Probabilità e le loro Applicazioni (GNAMPA) of the Istituto Nazionale di Alta Matematica (INdAM). 
	The authors are in debt with  Prof. Salvador Moll for many fruitful  comments and his crucial suggestion concerning the proof of Lemma \ref{lemma derivata salto nulla}. Finally, we wish to thank the anonymous referee for carefully reading this article and providing some useful comments.

\end{document}